\documentclass[a4paper,11pt]{amsart}
\usepackage[colorlinks, linkcolor=blue,anchorcolor=Periwinkle,
    citecolor=blue,urlcolor=Emerald]{hyperref}
\usepackage{mathabx}
\usepackage{cite}
\usepackage{qtree}
\usepackage{amssymb}
\usepackage{amsthm}
\usepackage{amsmath} 
\usepackage{mathtools}
\usepackage{graphicx}
\usepackage{adjustbox}

\usepackage[all,2cell]{xy}
\SelectTips{cm}{}
\usepackage[mathscr]{euscript}
\usepackage{listings}
\usepackage{young}
\usepackage{ytableau}
\usepackage{enumitem}
\usepackage{tikz}
\usetikzlibrary{arrows}
\usepackage{subcaption}
\usepackage{youngtab}
\usepackage{tikz-cd}
\usepackage{ifthen}
\usepackage{color}
\usepackage[normalem]{ulem}
\usetikzlibrary{arrows}
\usepackage{quiver}
\usetikzlibrary{automata}
\newtheorem{theorem}{Theorem}[section]

\newtheorem{definition}[theorem]{Definition}
\newtheorem{example}[theorem]{Example}
\newtheorem{lemma}[theorem]{Lemma}
\newtheorem{notation}[theorem]{Notation}
\newtheorem{proposition}[theorem]{Proposition}

\newtheorem{remark}[theorem]{Remark}

\newcommand{\cf}{\mathcal{F}}

\newcommand{\eat}[1]{}

\newcommand{\grass}[2]{Gr_{{#1},{#2}}}

\newcommand{\ko}{\: , \;}
\newcommand{\per}{\mathrm{per}}
\newcommand{\Si}{\Sigma}

\newcommand{\ol}[1]{\overline{#1}}
\newcommand{\ul}[1]{\underline{#1}}
\newcommand{\ten}{\otimes}

\renewcommand{\hat}[1]{\widehat{#1}}

\newcommand{\iso}{\xrightarrow{_\sim}}

\newcommand{\Hom}{\mathrm{Hom}}
\newcommand{\RHom}{\mathrm{RHom}}

\newcommand{\Ext}{\mathrm{Ext}}
\newcommand{\add}{\mathrm{add}}

\newcommand{\cok}{\operatorname{cok}}

\newcommand{\sn}{\operatorname{sgn}}
\newcommand{\supp}{\mathrm{supp}\,}

\newcommand{\llbracket}{[\![}
\newcommand{\rrbracket}{]\!]}

\newcommand{\cc}{{\mathcal C}}
\newcommand{\cd}{{\mathcal D}}

\newcommand{\cg}{{\mathcal G}}
\newcommand{\ch}{{\mathcal H}}

\newcommand{\cj}{{\mathcal J}}

\newcommand{\cp}{{\mathcal P}}

\newcommand{\cs}{{\mathcal S}}
\newcommand{\ct}{{\mathcal T}}

\newcommand{\cy}{{\mathcal Y}}

\newcommand{\eps}{\varepsilon}

\newcommand{\Ga}{\Gamma}

\newcommand{\Z}{\mathbb{Z}}
\newcommand{\C}{\mathbb{C}}

\newcommand{\Mod}{\mathrm{Mod}\,}
\renewcommand{\mod}{\mathrm{mod}\,}
\newcommand{\proj}{\mathrm{proj}\,}
\newcommand{\cm}{\mathrm{cm}}
\newcommand{\ind}{\mathrm{ind}}
\newcommand{\sub}{\mathrm{sub}}

\renewcommand{\tilde}[1]{\widetilde{#1}}

\makeatletter
\DeclareRobustCommand*\cal{\@fontswitch\relax\mathcal}
\makeatother

\newenvironment{magenta}
{\color{magenta}}
{}
  
\renewenvironment{magenta}
{}{}  

\begin{document} 

%\date{June 3, 2018} 

\title[$g$-vectors and $DT$-$F$-polynomials]{$g$-vectors and $DT$-$F$-polynomials for Grassmannians\\
via additive categorification}
\author{Sarjick Bakshi}
\address{Department of Mathematics and Statistics \\
Indian Institute of Technology Kanpur\\
208 016\\
India.
}
\email{sarjick91@gmail.com}
\author{Bernhard Keller}
\address{Universit\'e Paris Cit\'e and Sorbonne Université, CNRS,
IMJ-PRG, F-75013 Paris, France}

\email{bernhard.keller@imj-prg.fr}
\urladdr{https://webusers.imj-prg.fr/~bernhard.keller/}


\begin{abstract} We review $\Hom$-infinite Frobenius categorification of cluster algebras with coefficients
and use it to give two applications of Jensen--King--Su's Frobenius categorification  of the Grassmannian: 
1) we determine the $g$-vectors of the Pl\"ucker coordinates
with respect to the triangular initial seed and 2) we express the $F$-polynomials associated
with the Donaldson--Thomas transformation in terms of $3$-dimensional Young diagrams
thus providing a new proof for a theorem of Daping Weng. 
\end{abstract}

\maketitle

\tableofcontents

\newpage
\section{Introduction}

\begin{magenta}
Cluster algebras arising from Grassmannian varieties have served both as a rich source of combinatorial models 
and as a testing ground for categorification techniques following foundational work of Postnikov \cite{Postnikov06}, Scott  \cite{Scott06} and 
Geiss--Leclerc--Schr\"oer \cite{GeissLeclercSchroeer08}.
Jensen--King--Su \cite{JensenKingSu16} extended the scope of Geiss--Leclerc--Schr\"oer's framework by providing a
Frobenius categorification of the entire Grassmannian coordinate ring, thereby enabling a
categorical interpretation of all the Pl\"{u}cker coordinates and the corresponding cluster structure.  
In parallel, $g$-vectors and their extended versions have emerged in the general theory of
cluster algebras as fundamental invariants governing
positivity, canonical bases, and mutation phenomena.  The aim of this work is to bring the general
theory to bear on the special case of the Grassmannian cluster algebra by computing the extended \mbox{$g$-vectors} 
of all Pl\"{u}cker coordinates as categorical indices in the Jensen--King--Su model. Furthermore, we
illustrate the general theory of Donaldson--Thomas transformations by giving a new, representation-theoretic approach
to the computation of the Donaldson--Thomas $F$-polynomials associated with the Grassmannian
cluster algebra.
\end{magenta}

\subsection{The Grassmannian cluster algebra}
Let $\grass{k}{n}$ denote the \textit{Grassmannian} variety of $k$-subspaces in complex $n$-space and $\mathbb{C}[\grass{k}{n}]$ \ the homogeneous coordinate ring of the cone of the Grassmannian. Since the invention of cluster algebras by Fomin and Zelevinsky \cite{fomin2002cluster} and Scott's work \cite{Scott06}, the algebra $\mathbb{C}[\grass{k}{n}]$ played an important role as a non trivial example of a cluster algebra with coefficients. It is known as the Grassmannian cluster algebra.
Let us recall some milestones of its study: 
In \cite{fomin2003cluster}, Fomin and Zelevinsky classified the cluster-finite cluster algebras. In the same paper, they also show that the homogeneous coordinate ring of the Grassmannian $\grass{2}{n}$ is a cluster algebra with coefficients (see \cite[Proposition 12.6]{fomin2003cluster}) whose exchange relations come from the Pl\"{u}cker relations. In \cite{Scott06}, Scott generalised this result to all Grassmannians by showing that the homogeneous coordinate ring of $\grass{k}{n}$ admits a cluster algebra structure using the generalisation, due to Postnikov \cite{Postnikov06},  of double wiring arrangements called alternating wiring arrangements. 

\subsection{Frobenius categorification}
Geiss--Leclerc--Schr\"{o}er, in their fundamental paper \cite{GeissLeclercSchroeer08}, categorified the cluster algebra structure on an open affine cell of the Grassmannian using a subcategory $\sub\, Q_k$ of the module category of the preprojective algebra of type $A_{n-1}$ (more generally, they categorified open affine cells of arbitrary partial flag
varieties).  Jensen--King--Su \cite{JensenKingSu16} extended Geiss--Leclerc--Schr\"oer's
categorification to the homogeneous coordinate algebra
of the {\em whole} Grassmannian using equivariant Cohen--Macaulay modules\begin{magenta}.\end{magenta}
Let $R$ denote the complete local coordinate ring of the singularity $x^k = y^{n-k}$. The cyclic group $G$ of 
$n$th roots of unity naturally acts on $R$ by rescaling the coordinates. Jensen--King--Su \cite{JensenKingSu16} 
studied the homogeneous coordinate algebra $\mathbb{C}[\grass{k}{n}]$ endowed with
Scott's cluster structure using the category $\cm^G(R)$ of $G$-equivariant Cohen--Macaulay $R$-modules. 
They showed that the categories $\cm^G(R)$ and $\sub\, Q_k$ are stably equivalent. This allowed
them to transfer Geiss--Leclerc--Schr\"oer's results from $\sub\, Q_k$ to $\cm^G(R)$. In particular,
they deduced that the category $\cm^G(R)$ is stably $2$-Calabi--Yau and admits a 
cluster tilting object (also known as maximal one-orthogonal object, for the definition see section 
\ref{s.Introduction 3}). Building on  \cite{GeissLeclercSchroeer08}, they also obtained a
cluster character $CC: \cm^G(R) \to \mathbb{C}[\grass{k}{n}]$ which induces a bijection
from the set of indecomposable reachable rigid objects (up to isomorphism) of $\cm^G(R)$ to
the set of cluster variables in such a way that clusters correspond bijectively to reachable basic cluster-tilting objects. 

By Scott's work, each Pl\"ucker coordinate is a cluster variable and
Jensen--King--Su showed that the map $CC$ induces a bijection between 
(isomorphism classes of) rank one modules in $\cm^G(R)$ (i.e.~$G$-equivariant Cohen--Macaulay
modules whose underlying $R$-modules are free of rank $n$) and Pl\"ucker coordinates.
Many of their results were subsequently extended to quantum
cluster algebras \cite{JensenKingSu22} and generalized to positroid 
varieties, cf. \cite{CanakciKingPressland24}.

\subsection{$g$-vectors of Pl\"ucker coordinates and their categorification}
The $g$-vectors were introduced by Fomin--Zelevinsky in \cite{fomin2007cluster} motivated by 
Fock--Goncha\-rov's geometric approach to cluster algebras in \cite{FockGoncharov09}. The 
$g$-vectors play a crucial role in the proof of Fomin--Zelevinsky's positivity conjecture in the most general, 
skew-symmetrizable case by Gross--Hacking--Keel--Kontsevich \cite{gross2018canonical}. Assuming the existence of a \textit{reddening sequence} \cite{keller2011cluster, KellerDemonet20}, for cluster algebras with invertible coefficients,
they obtain a canonical basis of theta functions parametrised by all points of the $g$-vector lattice. 

 In \cite{dehy2008combinatorics}, Dehy and Keller introduced the notion of \textit{index} of an object of a Hom-finite 
 $2$-Calabi--Yau (CY)  triangulated category with respect to a cluster-tilting object. They showed that the $g$-vectors of the cluster monomials can be interpreted categorically as the indices of the reachable rigid objects. Their results have been generalised in \cite{FuKeller10}, where Fu and Keller studied the categorification of cluster algebras with coefficients and showed that different cluster monomials have different $g$-vectors, and that the $g$-vectors of any given cluster form a basis of the ambient lattice. In the full rank case, they deduced the linear independence of the cluster monomials (it was shown later in \cite{CerulliKellerLabardiniPlamondon13} that the cluster monomials are always linearly independent). 
 In this paper, we extend the definition of indices to the objects of a {\em stably $2$-Calabi--Yau Frobenius category} endowed with a cluster tilting object. Such categories were studied in depth in \cite{buan2009cluster}. 
 This allows us to interpret the  {\em extended $g$-vectors} of \cite{fomin2007cluster} as indices. 
 As an application, we  determine the extended $g$-vectors of the Pl\"ucker coordinates using Jensen--King--Su's
category $\cm^G(R)$ of $G$-equivariant Cohen--Macaulay modules. Concretely, in order to determine the extended 
$g$-vectors of the Pl\"ucker coordinates with respect to a given seed, it is enough to determine the indices of the rank one modules with respect to the corresponding cluster-tilting object. The seed we use is the {\em Le-diagram seed} 
for the Grassmannian, cf.~for example section 10.3.1 in  \cite{GeissLeclercSchroeer08} or section 3.1 in \cite{Fraser16}. Because of the shape of the corresponding quiver,  we also call it the {\em triangle seed} and denote it by $\boxslash_{k,n}$. This seed is different
from but mutation equivalent to Scott's \cite{Scott06} seed. Let us describe the result we obtain in more detail\begin{magenta}.\end{magenta}
 Let $I(k,n)$ denote the set of sequences 
 \[
 1 \leq a_1 < \ldots < a_k \leq n
 \]
 of integers. We know that the homogeneous coordinate algebra $\mathbb{C}[\grass{k}{n}]$ is generated 
as a $\mathbb{C}$-algebra by the Pl\"{u}cker coordinates $p_{w}$ with $w \in I(k,n)$. 
Let $w = (a_1,a_2,\ldots,a_k)$ and let ${\cal{Y}}_w$ be the Young diagram whose $i$th row from top
has length $a_{k-i+1}-(k-i+1)$. The following notions are implicit in  Lakshmibai--Weyman's \cite{lakshmibai1990multiplicities} combinatorial description of the singular locus of Schubert varieties:
We say that a box $b$ of ${\cal{Y}}_w$ is a {\em peak} if ${\cal{Y}}_w$ contains no boxes to the East and no boxes to the South of $b$. A box $b$ will be called a {\em valley} if there is a box to the South and a box to the East of $b$, 
but no box in the Southeast of $b$. 

\begin{theorem} (Theorem \ref{thm: g-vectors})
Let $I \in I(k,n)$. If $\mathcal{Y}_I$ is non empty, let $P$ denote the set of peaks and $V$ 
denote the set of valleys appearing in ${\cal Y}_I$. Then we have
\begin {itemize}
\item
If $I = (1,2,\ldots, k)$, then the $g$-vector of the Pl\"ucker coordinate $p_I$
with respect to the triangular seed of Theorem~\ref{ini.c}  is the basis vector $e_\emptyset$ of $\Z^m$ 
associated with the exceptional frozen vertex of $\boxslash_{k,n}$.
\item
If $I \neq (1,2,\ldots, k)$, then the $g$-vector of the Pl\"ucker coordinate $p_I$
with respect to the triangular seed of Theorem~\ref{ini.c} is given by
\[
\sum_{p \in P} e_p - \sum_{v \in V} e_v,
\]
where $e_p$ denotes the standard basis vector of $\Z^m$ associated with the vertex $p$ of
the quiver $\boxslash_{k,n}$. 
\end{itemize}
\end{theorem}

\subsection{Donaldson--Thomas invariants and their $F$-polynomials}
The link between cluster transformations and {\em Donaldson--Thomas ($DT$)} theory was discovered by 
Kontsevich--Soibelman \cite{kontsevichsoibelman2008}. They related the theory of $3$-Calabi-Yau categories with distinguished set of spherical generators with the theory of quivers with potential and constructed a non-commutative 
refined $DT$-invariant for such categories. Nagao in \cite{Nagao13}  was the first to give a complete 
dictionary between cluster combinatorics and data appearing in $DT$-theory.  He used it to provide new proofs for 
many of Fomin--Zelevinsky's conjectures, in particular the sign-coherence of $c$-vectors.
%They show under mutations the $DT$-invariants changes as composition of a monomial transformation and {\em quantum %dilogarithm}.  
Keller \cite{keller2011cluster} gave a combinatorial construction of Kontsevich--Soibelman's 
refined $DT$-invariant in terms of reddening sequences. A quiver may admit multiple reddening sequences. 
Each of these gives an expression of the refined $DT$-invariant as a product of quantum dilogarithms. 
By comparing these expressions, one can obtain 
many interesting  quantum dilogarithm identities, cf.~for example 
\cite{keller2011cluster,Nakanishi2011,Nakanishi11,KashaevNakanishi11,InoueIyamaKellerKunibaNakanishi13,InoueIyamaKellerKunibaNakanishi13a,Nakanishi15, GekhtmanNakanishiRupel17,Nakanishi18,Nakanishi21,Nakanishi24}. 

The {\em $DT$-transformation} of a cluster algebra is induced by a twist of the adjoint action of the corresponding
refined $DT$-invariant (when defined). \begin{magenta} The $DT$-transformation contains almost the same amount of 
information as the refined $DT$-invariant.\end{magenta}
The study of the $DT$-transformation for {\em cluster varieties} was lead by Goncharov--Shen in \cite{GoncharovShen2018},
where they study cluster $DT$-transformations on moduli spaces of $G$-local systems on surfaces using tropical points of 
cluster varieties. $DT$-transformations were studied for other important classes of varieties which admit cluster structures like
Grassmannians \cite{Weng21}, double Bott--Samelson cells \cite{Weng20, ShenWeng21} and braid varieties \cite{casals2024}. 

Much like $g$-vectors, {\em $F$-polynomials}, which are  certain integer polynomials, play an important role in the study of cluster algebras. They were introduced by Fomin--Zelevinsky in \cite{fomin2007cluster}. Inspired by the Caldero--Chapoton formula \cite{CalderoChapoton06}
Derksen--Weyman--Zelevinsky \cite{DerksenWeymanZelevinsky10} gave a representation-theoretic interpretation of these polynomials using 
representations of quivers with potential (which was instrumental in their proof of many of the conjectures made by Fomin--Zelevinsky in \cite{fomin2007cluster}). Whenever a quiver admits a reddening sequence, its (non-refined)
$DT$-invariant can be captured by certain $F$-polynomials, namely those associated with the cluster variables whose $g$-vectors 
are the opposite standard basis vectors (up to a permutation). These are precisely the cluster variables in the final seed obtained after 
a reddening sequence. We call these the {\em $DTF$-polynomials}, which is short for the $DT$-$F$-polynomials of
the title of this paper. It is immediate from Nagao's results in \cite{Nagao13} that whenever a quiver with non degenerate potential $(Q,W)$ admits a reddening sequence, the $i$th $F$-polynomial in the sequence $DTF_Q$ is given by $F_{I_i}$ where $I_i$ is the (right) module over the Jacobian algebra of $(Q,W)$ constructed as the injective hull of the simple module concentrated at the vertex $i$ of $Q$. 

Weng  \cite{Weng23} studied $DTF$-polynomials for several important classes of cluster algebras, notably
the coordinate algebras of varieties of triples of flags. He showed that the $DTF$-polynomials can be
computed as generating functions for ideals inside labeled posets. In the case of varieties of triples
of flags, each $DTF$-polynomial is obtained from the poset of $3D$ Young diagrams contained in a
rectangular cuboid. 

We give a simpler proof of Weng's result using a completely different approach\begin{magenta}.\end{magenta}
We study $DTF$-polynomials for the rectangular quiver $Q$ of the Grassmannian $Gr(k,n)$, 
cf.~Example~\ref{example Gr(4,9)}.
Let $L_m$ be the linearly  ordered set $1 < 2< \cdots < m$. 

\begin{theorem}[=Theorem~\ref{thm:Weng}, Weng \cite{Weng23}]
For a vertex $i=(p,q)$ of $Q$, the corresponding
$DTF$-polynomial is
\[
F_{I_i}(y)=\sum_K \prod_{(p',q',r')\in K} y_{p+p'-r', q+q'-r'}
\]
where $K$ ranges over the right ideals of the poset 
$L_r \times L_s \times L_t$ with $r=(n-k-1)-p$, $s=(k-1)-q$ and
$t=1 + \min(p-1, q-1)$.
\end{theorem}
In our representation-theoretic approach, the poset of right ideals in $L_r \times L_s \times L_t$ 
appears as the poset of graded submodules in an indecomposable injective module over the Jacobian algebra of
the quiver $Q$ endowed with its canonical potential. Notice that the right ideals of the poset $L_r \times L_s \times L_t$ 
are the $3D$ Young diagrams contained in the integral rectangular cuboid of side lengths $r,s$ and $t$,
as in Weng's description. We deduce that the non-zero coefficients of the $DTF$-polynomials equal $1$. 
It would be interesting to investigate the generalization of our results to partial flag varieties in
other types starting from their categorifications constructed by Geiss--Leclerc--Schr\"oer in 
\cite{GeissLeclercSchroeer08}.

\subsection{Related work}  
In this paper, the proof of Theorem~\ref{thm: g-vectors} on the $g$-vectors
of Pl\"ucker coordinates is based on a computation in the stable category of Cohen--Macaulay modules.
Alternatively, one can prove the theorem by constructing resolutions in the module category
itself in analogy with Baur--Bogdanic's construction of projective covers in \cite{BaurBogdanic17}.
This proof can be found in Lemma~4.9 of the first arXiv version \cite{BakshiKeller24v1} of this paper.
Yet another proof could be obtained using perfect matching modules and their projective
resolutions computed by Canakci--King--Pressland in \cite{CanakciKingPressland24}.

Let us point out that the computation of the $DT$-$F$-polynomials in Theorem~\ref{thm:Weng}
is related to the computation of cluster characters of twisted Pl\"ucker coordinates by 
Canakci--King--Pressland, who in \cite{CanakciKingPressland24} relate them to 
Marsh--Scott's combinatorial dimer partition functions \cite{MarshScott16} 
and similar formulas by Muller--Speyer \cite{MullerSpeyer17}.

%New references:
% \cite{BaurBogdanic17, CanakciKingPressland24, MarshScott16, MullerSpeyer17, FangGorskyPaluPlamondonPressland23}

\subsection*{Acknowledgments}

The authors would like to thank an anonymous referee for a careful reading of the manuscript and for several helpful suggestions that improved the exposition.
S.~B.~thanks the Universit\'e Paris Cit\'e and the Institut de math\'ema\-tiques de Jussieu--Paris Rive Gauche for the hospitality he
enjoyed during stays in June~2022 and May~2024. He would also like to thank Prof.~D.~Prasad, Prof.~S.~S.~Kannan and
Prof.~K.~V.~Subrahmanyam for their constant encouragement. 
During this work, he was partially supported by postdoc grants from Tata Institute of Fundamental Research, Mumbai and Indian Institute of Technology, Bombay. He dedicates this research work to the memory of his father Shri Tapan Kumar Bakshi. 

Both authors acknowledge support by the French ANR grant CHARMS 19 CE40 0017 headed by Yann Palu.
   
\section{Background on Cluster Algebras and $g$-vectors}\label{s.Introduction-1}
\begin{magenta}
In this section, we describe the categorical framework that underlies our interpretation
of extended $g$-vectors.  We begin by recalling the relevant notions concerning stably
$2$-Calabi–Yau Frobenius categories with cluster-tilting objects, including mutation
and the index of an object relative to a fixed cluster-tilting object.  In particular, we 
establish the compatibility of mutation of cluster-tilting objects
with the mutation of the indices of their indecomposable summands 
and show how these indices decategorify to extended $g$-vectors.
We then specialise to the Jensen–King–Su model for the Grassmannian, describing the
completed preprojective algebra of affine type $\widetilde{A}_{n-1}$, the boundary
algebra $B$, and the Frobenius category $\cm(B)$ of Cohen–Macaulay modules.
Within this categorical model, we recall how rank-one modules correspond to
Plücker coordinates and how non-crossing collections give rise to cluster-tilting
objects in the Grassmannian cluster category $\cm(B)$.

\end{magenta}

\subsection{From ice quivers to cluster algebras with coefficients} \label{ss:from-ice-quivers} \begin{magenta} In this subsection,
we recall the basic definitions of cluster algebras with coefficients and introduce extended $g$-vectors.  
We outline their behaviour under mutation and explain their role as combinatorial invariants of cluster variables.
The extended $g$-vectors will be compared with categorical indices in later sections.
\end{magenta}

A {\em quiver} $Q$ is a directed graph $(Q_0, Q_1, s, t)$ formed by a set of vertices $Q_0$, a set of arrows $Q_1$ and two maps 
$s: Q_1 \rightarrow Q_0$ and $t: Q_1 \rightarrow Q_0$ which take an arrow to its source and target respectively. A quiver $Q$ is {\em finite} if both $Q_0$ and $Q_1$ are finite. Let $Q_0 = \{1,2, \ldots, m\}$. 
An {\em ice quiver} is a quiver endowed with a subset $F$ of the set
$Q_0$ of its vertices. The vertices in $F$ are called {\em frozen}. We usually assume
that the frozen vertices are the vertices $r+1$, \ldots, $m$ for some $r\leq m$. 
We then say that $(Q,F)$ is an ice quiver of {\em type $(r,m)$}. The {\em principal part} of $Q$ is the full subquiver on the non frozen vertices. An arrow $\alpha$ is a {\em loop} if its source and target coincide. A {\em $2$-cycle} of $Q$ is a pair of distinct arrows $\beta$ and $\gamma$ such that $s(\beta) = t(\gamma)$ and $t(\beta) = s(\gamma)$. For two vertices $i$ and $j$ let $a_{ij}$ denote the number of arrows from $i$ to $j$. Let $b_{ij} = a_{ij} - a_{ji}$. We associate to each such quiver $Q$ the 
$m\times r$-matrix $\tilde{B}=\tilde{B}_Q$ whose $ij$-th entry is given by $b_{ij}$. 
It is called the {\em extended exchange matrix} of $Q$.
The {\em skew-symmetric} submatrix $B$ formed by the first $r$ rows is the 
{\em principal part} of $\tilde{B}$. 
Clearly, if $Q$ does not have loops nor $2$-cycles the matrix $\tilde{B}$ determines $Q$ up to the arrows between the
frozen vertices, which will play no role in this article.

Let $Q$ be an ice quiver of type $(r,m)$ without loops or $2$-cycles. We
recall from \cite{fomin2007cluster} how to construct the associated
cluster algebra.

Let $k$ be a  non-frozen vertex of $Q$. The {\em mutated quiver} $\mu_k(Q)$ is an ice quiver of type $(r, m)$ with the same vertex set and whose arrows can be obtained as follows:

\begin{itemize}
\item[(1)] for each subquiver $i \longrightarrow  k \longrightarrow j$, add a new 
arrow $i \longrightarrow j$;
\item[(2)] reverse all arrows with source or target $k$;
\item[(3)] remove the arrows in a maximal set of pairwise disjoint $2$-cycles.
\end{itemize}

Let 
\[ 
\sn(x) =  \left\{ \begin{array}{ll} -1 & \mbox{if $x < 0$;} \\
0 & \mbox{if $x=0$;}\\
1 & \mbox{if $x >0$.}
\end{array} \right.
\]
The mutated quiver $\mu_{k}(Q)$ corresponds to the mutated matrix 
$\mu_k(B)=(b'_{ij})$, whose coefficients $b'_{ij}$ are given by 
\[
b_{ij}' = \left\{ \begin{array}{ll} -b_{ij} & \mbox{if $i= k$ or $j=k$;} \\
b_{ij} + \sn(b_{ik})[b_{ik}b_{kj}]_{+}  & \mbox{otherwise,} 
\end{array} \right.
\]
where, for a real number $x$, we denote by $[x]_+$ the maximum between
$x$ and $0$.

Let $\mathbb{T}_r$ be the $r$-regular tree, where the $r$ edges emanating from each vertex are labeled by the numbers $1,2,\ldots,r$. 
Let us fix an initial vertex $t_0$ of $\mathbb{T}_r$.  To each vertex $t$, we associate an ice quiver $Q(t)$ as follows: 
\begin{itemize}
\item[(1)] we put $Q(t_0) = Q$; 
\item[(2)] whenever there is an edge labeled $k$ between two vertices $t$ and $t'$ in $\mathbb{T}_r$,
we put $Q(t') = \mu_{k}(Q(t))$.
\end{itemize}
The family of quivers $Q(t)$, where $t$ runs through the vertices of $\mathbb{T}_r$, is the {\em quiver pattern} associated with $Q$. 
The associated {\em matrix pattern} is the family of the matrices $\tilde{B}_{Q(t)}=\tilde{B}(t)$. Of course, it can be
defined directly from the matrix $\tilde{B}(t_0)$.

Let $x_1,x_2, \ldots x_m$ be $m$ indeterminates. To each vertex $t$ of $\mathbb{T}_r$, we associate a sequence called a {\em cluster}
of rational expressions $X_i(t)$, $1\leq i \leq m$, called {\em cluster variables}. They are defined recursively as follows:
\begin{itemize}
\item[(1)] $X_i(t_0) = x_i$, $1 \leq i \leq m$;
\item[(2)] $X_i(t) = x_i$, $r+1 \leq i \leq m$ for all $t$;
\item[(3)] whenever there is an edge labeled $k$ between two vertices $t$ and $t'$ in $\mathbb{T}_r$, we
define $X_i(t') = X_i(t)$ for all $i \neq k$, and $X_k(t')$ is determined by the {\em exchange relation}
\[ 
X_k(t)X_k(t') = \prod_{i \rightarrow k} X_i(t) +  \prod_{k \rightarrow j} X_j(t),
\]
where the first product is taken over the set of arrows with target $k$ and the second product over
the set of arrows with source $k$.
\end{itemize}
For a vertex $t$ of $\mathbb{T}_{r}$ , we denote by $X(t)$ the sequence $(X_1(t), \ldots, X_m(t))$.
The family $(Q(t), X(t))$, where $t$ runs through the vertices of $\mathbb{T}_r$, is called a 
{\em cluster pattern} with {\em initial seed} $(Q(t_0), X(t_0))$. Each pair $(Q(t), X(t))$ is called a {\em seed}. 
Following \cite{fomin2007cluster}, we define
the {\em cluster algebra $\mathcal{A}_Q$} associated with the cluster pattern $(Q(t), X(t))_{t\in \mathbb{T}_r}$
as the $\mathbb{C}[x_{r+1}, x_{r+2}, \ldots x_m]$-subalgebra of $\mathbb{C}(x_1, x_2, \ldots x_m)$ 
generated by all the cluster variables. A {\em cluster monomial} is a product of cluster variables lying in the same cluster.

 \subsection{Extended $g$-vectors}  \label{ss:extended g-vectors} We will give three constructions of the
 extended $g$-vector of a cluster variable. We fix an ice quiver $Q$ as above.
 
Historically, the $g$-vector of a cluster variable was first defined using a $\mathbb{Z}^m$-grading on a 
cluster algebra with principal coefficients (see \cite[\S6]{fomin2007cluster})\begin{magenta}.\end{magenta} Let $Q_{pr}$ be the quiver obtained from $Q$ by adding new frozen vertices $m+1, m+2, \ldots, 2m$ and new arrows $i \rightarrow i+m, 1\leq i \leq m$. We call the quiver $Q_{pr}$ the {\em principal extension} of $Q$. For example, we have
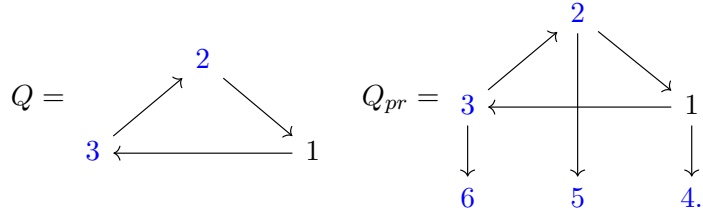
\[Q = \begin{tikzcd}
	& \color{blue}{2} \\
	\color{blue}{3} && 1
	\arrow[from=2-1, to=1-2]
	\arrow[from=1-2, to=2-3]
	\arrow[from=2-3, to=2-1]
\end{tikzcd} \quad Q_{pr} =
\begin{tikzcd}
	& \color{blue}{2} \\
	\color{blue}{3} && 1 \\
	\color{blue}{6} & \color{blue}{5} & \color{blue}{4.}
	\arrow[from=2-1, to=1-2]
	\arrow[from=1-2, to=2-3]
	\arrow[from=2-3, to=2-1]
	\arrow[from=2-3, to=3-3]
	\arrow[from=1-2, to=3-2]
	\arrow[from=2-1, to=3-1]
\end{tikzcd}\]
All the vertices coloured blue are frozen. 
By Cor.~6.2 of \cite{fomin2007cluster}, the cluster algebra ${\cal{A}}_{Q_{pr}}$ with principal coefficients is contained
in the algebra $\mathbb{C}[x_1^{\pm},x_2^{\pm},\ldots x_{m}^{\pm}, x_{m+1},\ldots, x_{2m}]$. 
Denote by $Q_u$ the quiver obtained from the ice quiver $Q$ by declaring the frozen vertices mutable and by $B_u$ the
corresponding exchange matrix. It is a skew-symmetric $m\times m$-matrix. We endow the ring 
$\mathbb{C}[x_1^{\pm},x_2^{\pm},\ldots x_{m}^{\pm}, x_{m+1},\ldots, x_{2m}]$ with a $\mathbb{Z}^m$-grading as follows: 
We declare 
\[ 
{\rm{deg}} (x_j) = e_j \quad \text{and } {{\rm{deg}}(x_{m+j})} = -B_{u} e_j, \text{ for }  1 \leq j \leq m. 
\]
Using \cite[Proposition 6.1]{fomin2007cluster} one checks that for each $t\in\mathbb{T}_r$ and each $1\leq i \leq r$,
 the cluster variables $X_{i}(t)$ of ${\cal{A}}_{Q_{pr}}$ are homogeneous for this grading. The degree is by definition the 
 {\em extended $g$-vector} of the cluster variable. It lies in $\Z^m$. 
 For $t\in\mathbb{T}_r$, the {\em $G$-matrix $G(t_0,t)$} with respect to the initial vertex $t_0$ is the matrix of size $m\times m$ 
whose columns are  the vectors $g_j(t)$ of $\Z^m$. By definition, the matrix $G(t_0, t_0)$ is the identity matrix.

For the second construction, we first need to introduce another family of vectors called 
\mbox{$c$-vectors}. 
Let $B(t_0)$ be the principal part of the
 exchange matrix $\tilde{B}(t_0)$ associated with the ice quiver $Q$. 
 Then the matrix pattern associated with the block matrix $[B(t_0), I_r]^T$
 consists of matrices of the form $[B(t), C(t)]^T$, where $B(t)$ is the principal part of $\tilde{B}(t)$ and $C(t)$ is
 an integer $r\times r$-matrix called the $c$-matrix associated with $Q$ and $t$. Its columns $c_i(t)$ are called
 the {\em $c$-vectors} at the vertex $t$. A fundamental theorem first proved in \cite{DerksenWeymanZelevinsky10} states that 
 each $c$-vector is non zero with entries which are either all non negative or all non positive (sign-coherence
 of the $c$-vectors). We define a non frozen vertex $i$ of the quiver $Q(t)$ to be {\em green} if the
 corresponding $c$-vector has all non negative coefficients; otherwise, it is defined to be {\em red}.
 Notice that this colouring depends not only on the cluster pattern but also on the choice of
 the initial vertex $t_0$. We define all frozen vertices to be {\em blue}.
 Let $e_1, e_2, \ldots ,e_m$ be the standard basis of the free abelian group of rank $m$. In the second construction,
the extended $g$-vectors
\[
g_i^{t_0}(t)=g_i(t)
\]
with respect to the initial vertex $t_0$ at a vertex $i$ of $Q(t)$ are obtained recursively as follows: 
\begin{itemize}
\item[(1)] $g_{i}(t_0) = e_i$, $1 \leq i \leq m$;
\item[(2)] whenever there is an edge labeled $k$ between two vertices $t$ and $t'$ in $\mathbb{T}_r$, we define $g_i(t') = g_i(t)$ for all $i \neq k$ and $g_k(t')$ is given by
\[ g_k(t') = 
 \begin{cases} -g_k(t) + \sum\limits_{k \rightarrow i} g_i(t), & \text{if $k$ is green in $Q(t)$ with respect to $t_0$;}  \\
 -g_k(t) + \sum\limits_{i \rightarrow k}  g_i(t), & \text{if $k$ is red in $Q(t)$ with respect to $t_0$,} 
\end{cases}
\]
where the sums are taken over the set of arrows with source (respectively, target) $k$.
\end{itemize}
Let us rewrite this definition in terms of the $g$-matrices: 
Let $\mathbb{I}_m$ denote the $m \times m$ identity matrix. 
For $\eps \in \{1, -1\}$, let $E_{k,\eps}(Q)$ denote the matrix 
of size $m\times m$ whose entries are given by
\[ (E_{k,\eps}(Q))_{ij} = 
\begin{cases}
1 & \text{$i \neq k$ and $j=i$}\\
0 & \text{$i \neq k$ and $j \neq i$}\\
-1& \text{$i =j = k$}\\
[-\eps b_{ik}]_+ & \text{$i \neq k$ and $j \neq i$} .
\end{cases}
\] 
Note that $E_{k,\eps}(Q)$ differs from $\mathbb{I}_m$ only in the $k$th column. 
Now we can reformulate the second construction as follows:
Whenever there is an edge labeled $k$ between two vertices $t$ and $t'$ in $\mathbb{T}_r$, 
we have 
\begin{equation} \label{gvectorscolor}
G(t_0,t') = G(t_0,t)E_{k,\eps}(Q(t)),
\end{equation}
where $\eps=1$ if $k$ is green in $Q(t)$ and $\eps=-1$ if $k$ is red in $Q(t)$.
The fact that this construction yields the same
$g$-matrices as the first one is proved as follows: First, in the above situation, the $k$th row
of the $g$-matrix $G(t_0, t)$ is non zero with all coefficients non negative or non positive by
Conjecture~1.3 of \cite{DerksenWeymanZelevinsky10}, proved in that paper. Moreover, the sign
of the $k$th row of $G(t_0, t)$ equals the sign of the $c$-vector $c_k(t)$ since the
$c$-matrix $C(t)$ is the inverse transpose of $G(t_0, t)$ by Theorem~1.2
of \cite{nakanishi2012tropical}. Thus, the above formula follows from Conjecture~1.6
of \cite{DerksenWeymanZelevinsky10}, proved in that paper.

We now give the third construction of the $g$-vectors respectively the $g$-matrices $G(t_1,t_2)$.
It uses induction on the distance between $t_1$ and $t_2$ in the regular tree $\mathbb{T}_r$.
Of course, we define the matrices $G(t,t)$ to be the identity matrices $I_m$. Now suppose
that $G(t_1, t_2)$ has been defined and that $t'_1$ is a vertex of the $r$-regular tree linked to $t_1$ 
by an edge labeled $k$. By Conjecture~1.3 of \cite{DerksenWeymanZelevinsky10}, proved in that
paper, the coefficients in a given row of the matrix $G(t_1, t_2)$
all have the same sign. Let us denote by $\eps\in\{1, -1\}$ the common sign of the coefficients in the
$k$th row. Then we construct $G(t_1', t_2)$ via
\begin{equation} \label{gvectorssign}
G(t'_1, t_2) = E_{k,\eps}(Q(t_1)) \, G(t_1, t_2).
\end{equation}
We can reformulate this definition as follows: Define $\phi: \Z^m \to \Z^m$ by
\[
\phi(v) = \begin{cases}
\phi_{+}(v) & \text{if $v = \sum x_i e_i$ with $x_{k} \geq 0$} \\
\phi_{-}(v) & \text{if $v = \sum x_i e_i$ with $x_{k} \geq 0$},
\end{cases} 
\]
where
\[ 
\phi_{+}(e_k) =  -e_k + \sum\limits_{i \rightarrow k}e_i 
\quad\mbox{and}\quad
\phi_{-}(e_k) = -e_k + \sum\limits_{k\rightarrow j} e_j .
\]
Then if $v$ is the $j$th column of $G(t_1, t_2)$ and $v'$ the $j$th column of $G(t_1', t_2)$, we
have $v'=\phi(v)$.  This formula is equivalent to Conjecture~7.12 in \cite{fomin2007cluster}, which is 
now proved (see \S 9,\cite{DerksenWeymanZelevinsky10}). 

\subsection{Indices in $\Hom$-infinite stably $2$-CY Frobenius categories}
\begin{magenta} In this subsection, we develop the categorical framework which we use in the sequel.
We introduce stably $2$-Calabi--Yau Frobenius categories, which provide a natural
setting for the categorification of cluster algebras with coefficients.  In the
subsections that follow, we define the index of an object with respect to a
cluster-tilting object, study its behaviour under mutation, and explain its
relationship with (extended) $g$-vectors. It is in this framework that we will 
later compute the $g$-vectors of the Plücker coordinates in a Grassmannian cluster
algebra.
\end{magenta}

A {\em Krull--Schmidt} category is an additive category where indecomposable objects have local endomorphism rings and
each object decomposes into a finite direct sum of indecomposable objects (which are then unique up to isomorphism
and permutation). Recall that the endomorphism ring $E$ of any object in a 
Krull-Schmidt category is semiperfect, i.e.\ each finitely generated $E$-module has a projective cover
(see \cite[Corollary 4.4]{krause2015krull}). An object in a Krull--Schmidt category is {\em basic} if its indecomposable
summands occur with multiplicity at most one. A basic object $X$ is determined up to isomorphism by the full 
additive subcategory ${{\rm{add}}}(X)$ whose objects are the direct factors of finite direct sums of copies of $X$. 

Let $K$ be an algebraically closed field. 
Let $\mathcal{C}$ be a $K$-linear triangulated Krull--Schmidt category 
with suspension functor $\Sigma$. The category $\mathcal{C}$ is {\em $2$-Calabi-Yau}, if 
it is $\Hom$-finite and the square of the 
suspension functor is a Serre functor for $\cal{C}$ so that we have bifunctorial isomorphisms
\[ 
D{\cal{C}}(X, Y) \xrightarrow{\sim} {\cal{C}}(Y, \Sigma^{2} X) ,
\]
where $D$ denotes the duality functor $\Hom_K(?, K)$ over the ground field $K$. 

Let $\mathcal{C}$ be an exact category. It is said to have {\em enough projectives} if for each $X \in \cal{C}$, there is a 
deflation $P \rightarrow X$ with a projective $P$. Dually, $\cal{C}$ is said to have {\em enough injectives}  if for each $X \in \cal{C}$, 
there is an inflation $X \rightarrow I$ with an injective $I$. An exact category $\mathcal{C}$ is {\em Frobenius} if it 
has enough projectives and enough injectives and the class of the projective objects coincides with that of the injective objects. 
We recall from \cite{keller1996derived} that the {\em stable category $\ul{\mathcal{C}}$} associated with a Frobenius
category $\mathcal{C}$ has the same objects as $\mathcal{C}$. A morphism of $\underline{\mathcal{C}}$ is the equivalence class 
$\bar{f}$ of a morphism $f: A \rightarrow B$ of $\mathcal{C}$ modulo the subgroup of morphisms factoring through 
an injective of $\mathcal{C}$. The stable category $\underline{\mathcal{C}}$ of a Frobenius category $\mathcal{C}$ is a 
triangulated category (see, \cite[Theorem 9.4]{happel1987derived}). 
Let $\mathcal{C}$ be a $K$-linear Krull--Schmidt Frobenius category. We say that $\mathcal{C}$ is
{\em stably $2$-Calabi--Yau} ($2$-CY for short) if $\ul{\mathcal{C}}$ is $2$-Calabi--Yau (hence,
in particular $\Hom$-finite).

%%%%%

In the sequel, we assume that all $K$-linear categories under consideration
are moreover enriched over the symmetric monoidal category of
pseudocompact vector spaces, cf. section~4 of \cite{van2015calabi}.

Let $\mathcal{C}$ be a $K$-linear category which is either triangulated $2$-CY or exact Frobenius
stably $2$-CY. Recall that, if $\mathcal{C}$ is triangulated, for objects $X,Y$ of $\cal{C}$ 
and any integer $i$, one defines
\[ 
{\rm{Ext}}^i_{\cal{C}} (X,Y) = {\mathcal{C}}(X, \Sigma^i Y).
\]
An object $X$ of $\cal{C}$ is {\em rigid} if 
\[ 
{\rm{Ext}}^1_{\mathcal{C}} (X,X) = 0.
\] 
A {\em cluster tilting object} is a basic object $T$ of $\cal{C}$ such that $T$ is rigid and each object $X$ satisfying 
${\rm{Ext}}_{\mathcal{C}}^1(T,X) = 0$ belongs to ${\rm{add}}(T)$. For a cluster tilting object $T$, we write $Q_T$
 for the quiver of ${\rm{End}}_{\cal{C}}(T)$. Notice that by our assumptions, the endomorphism algebra ${\rm{End}}_{\cal{C}}(T)$
 is a pseudocompact algebra (sometimes even finite-dimensional) so that its quiver is well-defined.
 The quiver $Q_T$ is also called the {\em endoquiver} of $T$. It is constructed as follows. 
 For two indecomposable objects $T'$ and $T''$ of $\add(T)$, let $ {\rm{rad}}(T',T'')$ be the space of non isomorphisms from $T'$ to $T''$ so that $ {\rm{rad}}$ is the radical ideal of the category $\add(T)$. 
Let $T_1, T_2, \ldots, T_m$ denote representatives of the isomorphism classes of the indecomposable 
objects of ${\rm{add}}(T)$. By definition, the vertices of $Q_T$ are the integers $1$, \ldots, $m$
corresponding to the indecomposables $T_i$ and 
 the number of arrows from $i$ to $j$ is the dimension of the {\em space of irreducible morphisms}
\[ 
{\rm{irr}}(T_i, T_j) = {\rm{rad}}(T_i,T_j)/{\rm{rad}}^2(T_i,T_j).
\] 
We always assume it is finite. If $\mathcal{C}$ is a Frobenius category, we define the {\em frozen subquiver} of
$Q_T$ to be the full subquiver on the vertices $i$ such that $T_i$ is projective-injective. Thus, $Q_T$ becomes
an {\em ice quiver}. 

Combinatorial cluster structures on stably $2$-Calabi--Yau Frobenius categories were studied in \cite{buan2009cluster}. 
We briefly recall their description. Let $Q$ be an ice quiver.
Let $(\mathcal{C}, T)$ be a $2$-{\em Calabi--Yau realization} of $Q$, i.~e.~a pair consisting of
a stably $2$-CY Frobenius category $\cc$ and a cluster-tilting object $T$ in $\cal{C}$ such that
\begin{itemize}
\item[a)] the endoquiver $Q_{T}$ is isomorphic, as an ice quiver, to $Q$ and 
\item[b)] the cluster-tilting subcategories of $\cal{C}$
determine a cluster structure on $\cal{C}$ in the sense of section~I.1
of \cite{buan2009cluster}. 
\end{itemize}
By Theorem~I.1.6 of [loc.~cit.], condition b) holds if no cluster-tilting
object of $\cal{C}$ has loops or $2$-cycles in its quiver. By Prop.~2.19 (v) of \cite{GeissLeclercSchroer11},
this holds for many stably $2$-CY categories occuring in Lie theory.

Let $T = \oplus_{i=1}^m T_i$ be the decomposition of the basic object $T$ into indecomposables. 
Let $T_1, T_2, \ldots, T_r$ denote the non-projective indecomposable summands and 
$T_{r+1}, \ldots, T_m$ the projec\-tive-injective indecomposable summands. 
The non-projective indecomposable summands of $T$ correspond to the non frozen initial cluster variables 
and the projective-injective indecomposable summands correspond to the frozen variables. 
The mutation at a non frozen vertex $k$ leads to the cluster tilting object 
\[
\mu_k(T) = T_k^* \oplus \bigoplus_{i\neq k} T_i .
\] 
Recall from Lemma~2.2 of \cite{FuKeller10}  that under the assumption b), the spaces
$\Ext^1_{\mathcal{C}}(T_k, T_k^*)$ and $\Ext^1_{\mathcal{C}}(T_k^*, T_k)$ are one-dimensional 
so that we have non split {\em exchange conflations}
\[
\begin{tikzcd}
	{T_k^*} & E & {T_k}
	\arrow[tail, from=1-1, to=1-2]
	\arrow[two heads, from=1-2, to=1-3]
\end{tikzcd} \quad\mbox{and}\quad
\begin{tikzcd}
	{T_k} & E' & {T_k^*}
	\arrow[tail, from=1-1, to=1-2]
	\arrow[two heads, from=1-2, to=1-3]
\end{tikzcd}
\]
whose middle terms are unique up to isomorphism.
\begin{magenta}
We restate the following lemma from \cite{buan2009cluster} which compares the cluster tilting objects of $\cal{C}$ and $\underline{\cal{C}}$.  

\begin{lemma}\cite[Lemma II.1.4]{buan2009cluster}
Let $\cal{C}$ be a stably $2$-CY Frobenius category. Then $T$ is a cluster-tilting object in $\cal{C}$ if and only if 
its image $\underline{T}$ in $\underline{\mathcal{C}}$ is a cluster-tilting object.
\end{lemma}

\subsection{A categorical model for Grassmannian cluster algebras} \label{sec:Grassmannian}
In this subsection, we recall the Jensen--King--Su model for Grassmannian cluster
algebras.  We describe the boundary algebra $B$, the Frobenius category $\cm(B)$,
its graded enhancement, and the indecomposable rank--one modules $L_I$ associated
with $k$-element subsets of $\{1, \ldots, n\}$.  We also recall the description of Plücker clusters and the
construction of the corresponding cluster–tilting objects.  These ingredients will be
fundamental for computing indices and extended $g$-vectors in later sections.

Let $K=\C$. Let $0<k<n$ be integers. 
Let $\Pi$ be the completed preprojective algebra of affine type $\tilde{A}_{n-1}$. Thus, for
$n=6$ the algebra $\Pi$ is the completed path algebra of the quiver
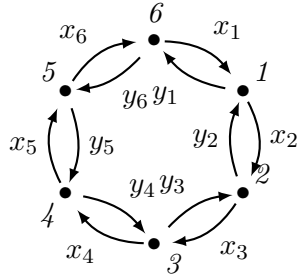
\begin{center}
\begin{tikzpicture}[scale=0.90,baseline=(bb.base)]
\path (0,0) node (bb) {};
\newcommand{\radius}{1.5cm}
\foreach \j in {1,...,6}{
  \path (90-60*\j:\radius) node[black] (w\j) {$\bullet$};
  \path (150-60*\j:\radius) node[black] (v\j) {};
  \path[->,>=latex] (v\j) edge[black,bend left=25,thick] node[black,auto] {$x_{\j}$} (w\j);
  \path[->,>=latex] (w\j) edge[black,bend left=20,thick] node[black,auto] {$y_{\j}$}(v\j);
}
\draw (90:\radius) node[above=3pt] {6};
\draw (150:\radius) node[above left] {5};
\draw (210:\radius) node[below left] {4};
\draw (270:\radius) node[below right] {3};
\draw (330:\radius) node[above right] {2};
\draw (30:\radius) node[above right] {1};
\end{tikzpicture}
\end{center}
subject to the $n$ relations $xy=yx$. Let $B$ be the quotient of $\Pi$ by the closed
ideal generated by the $n$ relations $x^k-y^{n-k}$. The algebra $\Pi$ is Noetherian
of global dimension $2$, cf. \cite{CrawleyBoeveyHolland98}, 
and $B$ is Noetherian of infinite global dimension. Moreover,
$B$ is $1$-Iwanaga--Gorenstein, i.e. we have
\[
\Ext_B^i(M,B)=0
\]
for all $i>1$ and all finitely generated (right) $B$-modules $M$, cf. Cor.~3.4 of \cite{JensenKingSu16} and
its proof. Let $\cm(B)$ denote the category of finitely generated (maximal) Cohen-Macauley $B$-modules,
i.e.~finitely generated $B$-modules $M$ such that
\[
\Ext_B^i(M,B)=0
\]
for all $i>0$. This is the category denoted by $\cm^G(\hat{R})$ in \cite{JensenKingSu16}.
Since $B$ is Gorenstein, the category $\cm(B)$ is a Frobenius category, whose
projective-injectives are the projectives of $\mod B$. It is shown in section~4 of \cite{JensenKingSu16}
that the associated stable category is $\Hom$-finite and $2$-Calabi--Yau as a triangulated
category. We also know from section~3 of [loc.~cit.] that the center $Z$ of $B$ is isomorphic
to the power series algebra $\C[[t]]$ by the map sending $t$ to $xy$. Each module
in $\cm(B)$ is finitely generated free over $Z$. Thus, if $L$ and $M$ are in $\cm(B)$, the
subspace $\Hom_A(L,M) \subseteq \Hom_Z(L,M)$, which is of finite codimension,
naturally becomes a pseudo-compact vector space and clearly this defines
the required enrichment on $\cm(B)$. Finally, the category $\cm(B)$ is Krull-Schmidt
because $B$ is a Noetherian quotient of a completed path algebra.

Let us recall the graded version of the category $\cm(B)$ from section~6
of \cite{JensenKingSu16}: We define the arrows $x_i$ of the quiver of the
completed preprojective algebra of type $\tilde{A}_{n-1}$ to be of degree $e_1\in \Z^2$ and
the arrows $y_j$ to be of degree $e_2 \in \Z^2$, where $e_1$ and $e_2$ are the vectors
of the standard basis in $\Z^2$. This defines a $\Z^2$-grading on the pseudocompact
completed preprojective algebra. The elements $x^k-y^{n-k}$ are not homogeneous for
this grading but become homogeneous for the induced grading by the
group $\Ga^\vee$ defined as the quotient of $\Z^2$ by the subgroup generated by $k e_1 - (n-k) e_2$.
Thus, the boundary algebra $B$ inherits a $\Ga^\vee$-grading. We have a surjective homomorphism
$\Ga^\vee \to \Z/n\Z$ taking the class of an element $(r,s)$ to the class of the difference $r-s$.
The corresponding $\Z/n\Z$-grading on $B$ corresponds to the decomposition of $B$ into
the sum of the $B e_i$, where $e_i$ is the idempotent given by the lazy path at the vertex $i$, $0\leq i<n$. 
The kernel of the surjective morphism $\Ga^\vee \to \Z/n\Z$ is free of rank one generated by the class
of $e_1+e_2$. This reflects the fact that the $Be_i$ admit $\Z$-gradings as modules
over the $\Z$-graded pseudocompact algebra $Z=\C[[t]]$, the center of $B$ generated
by $t=\sum_i x_i y_i$, which is of degree $e_1+e_2$. 
Let us denote by $\cm^{gr}(B)$ the category of $\Ga^\vee$-graded
pseudocompact $B$-modules. This is the category denoted by $CM_\Ga(\ol{R})$ in
section~6 of \cite{JensenKingSu16}. We have a functor forgetting the grading
\[
\cm^{gr}(B) \to \cm(B)
\]
and, as shown in Lemma~6.2 of \cite{JensenKingSu16}, every rigid module in $\cm(B)$
lifts to $\cm^{gr}(B)$ and the lift is unique up to a grading shift by a multiple of the degree
of $t$ if the module is indecomposable.

 A module $M\in \cm(B)$ is {\em of rank one} if each $Z$-module $M e_i$, $1\leq i \leq n$, is free of rank one. As shown
in Prop.~5.2 of \cite{JensenKingSu16}, the isomorphism classes of rank one modules are in bijection
with the $k$-element subsets $I$ of the set $\{1, \ldots, n\}$. The bijection sends a $k$-element
subset $I$ to the isomorphism class of the $B$-module $L_I$ such that $L_I e_i=Z$ for all 
$1\leq i\leq n$ and the arrows act as follows: An arrow $x: i \to i+1$ acts by multiplication by
$t$ if $i\in I$ and by $1$ if $i\not\in I$; an arrow $y: i+1 \to i$ acts by multiplication by $1$ if
$i\in I$ and by $t$ if $i\not\in I$. We call $L_I$ the {\em Jensen--King--Su module}
associated with $I$.
The cluster-tilting objects we will consider are sums of
certain rank one modules. Two $k$-element subsets $I$ and $J$ of $\{1, \ldots, n\}$ are
{\em non-crossing} (Def.~3 of \cite{OhPostnikovSpeyer15}) if there are no cyclically ordered
elements $a, b, c, d$ of $\{1, \ldots, n\}$ such that the elements $a$ and $c$ belong to $I\setminus J$ and
the elements $b$ and $d$ belong to $J\setminus I$. By Prop.~5.6 of \cite{JensenKingSu16}, this
happens if and only if we have 
\[
\Ext^1_B(L_I, L_J)=0.
\]
A {\em Pl\"ucker cluster} is a maximal
collection of pairwise non-crossing $k$-element subsets $I$ of $\{1, \ldots, n\}$. As conjectured
by Scott \cite{Scott06} and proved by Oh--Postnikov--Speyer \cite{OhPostnikovSpeyer15}, the
Pl\"ucker clusters are exactly the collections of $k(n-k)+1$ pairwise non-crossing $k$-element
subsets of $\{1, \ldots, n\}$. Using this and the results of Geiss--Leclerc--Schr\"oer 
\cite{GeissLeclercSchroeer08}, one
deduces (cf.~Remark~5.6 of \cite{JensenKingSu16}) that each Pl\"ucker cluster $\cp$
yields the cluster-tilting object
\[
T_\cp= \bigoplus_{I\in\cp} L_I.
\]

As shown in Theorem~4.5 of \cite{JensenKingSu16}, the stable category of the
Frobenius category $\cm(B)$ is triangle equivalent to the
stable category of the category $\sub(Q_k)$ of 
\cite{GeissLeclercSchroeer08}. Therefore, it follows
from Prop.~2.19 (v) of \cite{GeissLeclercSchroer11} that the endoquivers of its
cluster-tilting objects do not have loops or $2$-cycles. By Theorem~I.1.6 of
\cite{buan2009cluster}, we deduce that the cluster-tilting objects of $\cm(B)$
determine a cluster structure. Assume now that we have $n=4$ and $k=2$. 
Then, up to isomorphism, there are exactly six indecomposable rank one modules. Among
these, only $L_{13}$ and $L_{24}$ are non projective. The quiver of the category of
rank one modules looks as follows:
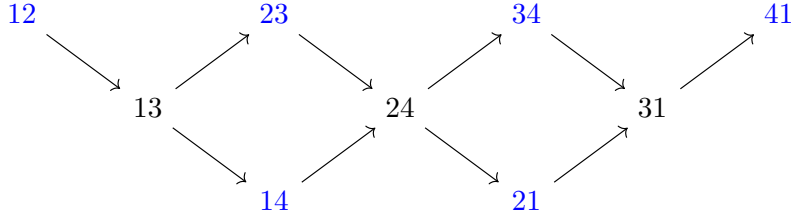
\[
\begin{tikzcd}
	{\color{blue}{12}} && {\color{blue}{23}} && {\color{blue}{34}} && {\color{blue}{41}} \\
	& 13 && 24 && 31 \\
	&& {\color{blue}{14}} && {\color{blue}{21}}
	\arrow[from=1-1, to=2-2]
	\arrow[from=2-2, to=3-3]
	\arrow[from=2-2, to=1-3]
	\arrow[from=1-3, to=2-4]
	\arrow[from=3-3, to=2-4]
	\arrow[from=2-4, to=3-5]
	\arrow[from=3-5, to=2-6]
	\arrow[from=1-5, to=2-6]
	\arrow[from=2-4, to=1-5]
	\arrow[from=2-6, to=1-7]
\end{tikzcd}
\]

Up to isomorphism, there are exactly two basic cluster tilting objects one having
$L_{13}$ and one having $L_{24}$ as a direct factor. These modules are linked
by the exchange conflations
\[
\begin{tikzcd}
{13} & {23\oplus 14} & {24.} 
	\arrow[tail, from=1-1, to=1-2]
	\arrow[two heads, from=1-2, to=1-3]
\end{tikzcd} \quad\mbox{and}\quad
\begin{tikzcd}
{24} & {34\oplus 21} & {31.} 
	\arrow[tail, from=1-1, to=1-2]
	\arrow[two heads, from=1-2, to=1-3]
\end{tikzcd} 
\]

\end{magenta}

\subsection{Index of a rigid object in Frobenius category} 
Let $\cal{C}$ be a stably $2$-CY Frobenius category that admits a cluster-tilting object $T$ and let $\underline{T}$
be its image in $\underline{\cal{C}}$. Let ${\cal{T}} = {\rm{add}}(T)$ and $\underline{\cal{T}} = {\rm{add}}(\underline{T})$. 
Recall that a {\em right ${\cal{T}}$-approximation} of an object  $X \in \cal{C}$ is a morphism $T_X \rightarrow X $
with $T_X \in \cal{T}$ such that each morphism $T' \rightarrow X$ with $T' \in \cal{T}$ factors through $T_X$.
Under our assumptions, right $\mathcal{T}$-approximations exist for all objects $X$  because $\mathcal{C}$ 
has enough projectives and $\underline{\mathcal{C}}$ is Hom-finite. Dually, left $\mathcal{T}$-approximations exist.
A right $\mathcal{T}$-approximation $T_X \to X$ is called {\em minimal} if 
\[ 
{\cal{C}}(T,T_X) \rightarrow {\cal{C}}(T,X) 
\] 
is a projective cover for the ${\rm{End}}(T)$-module ${\cal{C}}(T,X)$.
Since each object of $\mathcal{C}$ admits a left $\mathcal{T}$-approximation and a right $\mathcal{T}$-approximation,
the subcategory $\cal{T}$ is an example of a {\em cluster tilting subcategory} \cite{keller2008acyclic} or a 
{\em maximal $1$-orthogonal subcategory} of $\cal{C}$ in the sense of Iyama \cite{iyama2007higher}.

\begin{proposition} \cite[Proposition 4]{keller2008acyclic}
\label{rightapproximationprop}
For each   $X \in \cal{C}$ there is a conflation 
\begin{equation}
\label{eqn:1}
\begin{tikzcd}
	{T_1} & T_0 & {X}
	\arrow[tail, from=1-1, to=1-2]
	\arrow[two heads, from=1-2, to=1-3]
\end{tikzcd}
\end{equation}

such that $T_0$ and $T_1$ belong to $\cal{T}$ with the map  $T_0  \overset{h}{\longrightarrow} X$ a minimal right ${\cal{T}}$-approximation.
\end{proposition}

We recall that the (split) Grothendieck group $K_0(\cal{A})$ of an additive category $\cal{A}$ is the quotient of the free group on the isomorphism classes $[A]$ of objects $A$ of $\cal{A}$ by the subgroup generated by the elements of the form 
\[ [A_1\oplus A_2]-[A_1]-[A_2]. \]

Let $K_0(\cal{T})$ (respectively, $K_0(\underline{\cal{T}}$)) denote the Grothendieck group of the additive category $\cal{T}$ (respectively, $\underline{\cal{T}}$). It is isomorphic to the free abelian group on the isomorphism classes of the indecomposable objects of $\cal{T}$ (respectively, $\underline{\cal{T}}$). For an object $X \in \cal{C}$ admitting a conflation
\[ 
\begin{tikzcd}
	{T_1} & T_0 & {X}
	\arrow[tail, from=1-1, to=1-2]
	\arrow[two heads, from=1-2, to=1-3]
\end{tikzcd}
\] 
with $T_1, T_0 \in \cal{T}$, we put
\[ 
{\rm{ind}}_{\cal{T}}(X) = [T_0] - [T_1] \in K_0(\cal{T}).
\]

Let us show that this is well-defined. Let $B = {\rm{End}}_{\cal{C}}(T)$. Let ${\rm{Mod}}(B)$ denote the category of all right $B$-modules. Let ${\rm{per}}(B)$ denote the perfect derived category of $B$, which is the full subcategory of the unbounded derived category of ${\rm{Mod}} (B)$ whose objects are quasi-isomorphic to
complexes of finitely generated projective $B$-modules. The functor 
\[ {\cal{C}}(T,?): {\cal{C}} \rightarrow {\rm{Mod}}(B) \]
induces an equivalence from ${\rm{add}}(T)$ to the full subcategory ${\rm{proj}}(B)$ of finitely generated projective $B$-modules (see \cite[Proposition 2.3]{krause2015krull}). We apply ${\cal{C}}(T,?)$ to (\ref{eqn:1}) and obtain the following 
\begin{equation}
\label{eqn:2}
0 \longrightarrow {\cal{C}}(T,T_1) \longrightarrow {\cal{C}}(T, T_0) \longrightarrow {\cal{C}}(T,X) \longrightarrow {\rm{Ext}}_{\cal{C}}^1(T,T_1). 
\end{equation}

Since we have ${\rm{Ext}}^1_{\cal{C}}(T,T_1) = 0$, this sequence gives a projective resolution of the $B$-module ${\cal{C}}(T,X)$ and ${\cal{C}}(T,X)$ belongs to ${\rm{per}}(B)$. We have canonical isomorphisms 
\[ K_0({\rm{per}}(B)) \simeq K_0({\rm{proj}}(B)) \simeq K_0({\cal{T}}). \]

Under these, the class of ${\cal{C}}(T,X)$ is mapped to the class of ${\rm{ind}}_{\cal{T}}(X)$, which is therefore independent of the choice of the conflation (\ref{eqn:1}).

\begin{remark} \label{rk:no projective summands}
We claim that we can choose the conflation (\ref{eqn:1}) such that $T_1$ does not have non-zero projective summands. Indeed, since ${\cal{C}}$ is Krull-Schmidt, we can write $T_1 = P\oplus T'_1 $ where $P$ is projective and $T'_1$ has no non zero projective summands. Then the composition 
\[ P \xhookrightarrow{} T_1 \rightarrow T_0  \]
is an inflation. Since $P$ is also injective it is the inclusion of a direct factor. So we get a split exact sequence of conflations

\[\begin{tikzcd}
	P & P & 0 \\
	{T_1} & {T_0} & X \\
	{T'_1} & {T'_0} & X
	\arrow[Rightarrow, no head, from=1-1, to=1-2]
	\arrow[two heads, from=1-2, to=1-3]
	\arrow[from=1-3, to=2-3]
	\arrow[Rightarrow, no head, from=2-3, to=3-3]
	\arrow[two heads, from=2-2, to=2-3]
	\arrow[tail, from=2-1, to=2-2]
	\arrow[from=2-1, to=3-1]
	\arrow[tail, from=3-1, to=3-2]
	\arrow[from=3-2, to=3-3]
	\arrow[from=2-2, to=3-2]
	\arrow[hook', from=1-1, to=2-1]
	\arrow[tail, from=1-2, to=2-2]
\end{tikzcd} ,
\]
which proves the claim.
\end{remark}

\begin{proposition} \label{prop: rigid factor in common}
If $X \in {\cal{C}}$ is rigid, the objects $T_0$ and $T_1$ do not have an indecomposable direct factor in common.
\end{proposition}
\begin{proof}

We recall from \cite{happel1987derived} that the conflation 
\[
\begin{tikzcd}
	{T_1} & {T_0} & X
	\arrow["g", tail, from=1-1, to=1-2]
	\arrow["h", two heads, from=1-2, to=1-3]
\end{tikzcd}
\]
in ${\cal{C}}$ gives rise to a triangle $T_1 \overset{\underline{g}}\longrightarrow T_0 \overset{\underline{h}}\longrightarrow X \longrightarrow \Sigma T_0$ in $\underline{\cal{C}}$ where $T_0,T_1 \in \underline{\cal{T}}$. We can also choose $T_1$ such that $T_1$ does not have a projective-injective summand in ${\cal{C}}$. So if $U$ appears as a direct summand of $T_0$ and $T_1$ in ${\cal{C}}$, $U$ cannot have a non zero projective-injective summand.  Since $X$ is rigid and $\underline{h}$ is a minimal right $\underline{\cal{T}}$- approximation we know from \cite{dehy2008combinatorics} that $T_0$ and $T_1$ cannot have a common factor in $\underline{\cal{T}}$. So $U$ cannot have an indecomposable non-projective summand either, proving our claim.
\end{proof}
%%%%

\begin{theorem}
\label{thm:index determines rigid object}
Two rigid objects of ${\cal{C}}$ are isomorphic if and only if their indices are equal.
\end{theorem}

\begin{proof} Let $X_1$ and $X_2$ be two rigid objects of ${\cal{C}}$. 
We decompose them as $X_i= X_i' \oplus P_i$, $1\leq i \leq 2$,
where $P_i$ is projective-injective and $X'_i$ does not have non zero projective summands. 
\begin{magenta} Clearly, the objects $X'_1$ and $X'_2$ have the same index
with respect to $\ul{\ct}$ in the stable category $\ul{\cc}$. Since the $X'_i$ are rigid,
it follows from Theorem~2.3 of
\cite{dehy2008combinatorics} that $X'_1$ is isomorphic to $X'_2$ in $\ul{\cc}$ and hence in $\cc$.
Since $X_1 = X'_1\oplus P_1$ and $X_2=X_2'\oplus P_2$ have the same index
and $X_1'$ is isomorphic to $X'_2$, it follows that $P_1$ and $P_2$ have the
same index. Since $\ct$ is a Krull--Schmidt category, we obtain that
$P_1$ and $P_2$ are isomorphic. As a consequence, the objects
$X_1$ and $X_2$ are isomorphic.
\end{magenta}
\end{proof}

Let $\ct=\add(T)$ as above. Let $1 \leq k \leq r$ and assume that $T'$ is obtained from $T$
by mutation at the indecomposable non projective summand $T_k$.
We recall that the exchange conflations are given by
\[
\begin{tikzcd}
	{T_k^*} & E & {T_k}
	\arrow[tail, from=1-1, to=1-2]
	\arrow[two heads, from=1-2, to=1-3]
\end{tikzcd} \quad\mbox{and}\quad
\begin{tikzcd}
	{T_k} & E' & {T_k^*.}
	\arrow[tail, from=1-1, to=1-2]
	\arrow[two heads, from=1-2, to=1-3]
\end{tikzcd}
\]
\eat {Let \rm{indec}(\cal{X}) denotes the set of isomorphism classes of indecomposables appearing in a cluster tilting object $\cal{X}$. Then,
\[ \rm{indec}(\cal{T'}) = (\rm{indec}({\cal{T}}) \backslash {T_k}) \bigcup {T_k^*}. \]}

Following \cite{dehy2008combinatorics}, we define two linear maps 
\[ \phi_{+} : K_0({\cal{T}}) \rightarrow K_0({\cal{T'}}) \text{  and  } \phi_{-}:K_0({\cal{T}}) \rightarrow K_0({\cal{T'}}). \]
as follows,

\[\phi_{+}([T_k]) =  [E] - [T_k^*] ; \]
\[\phi_{-}([T_k])=  [E'] - [T^*_k]  ;\]
\[ \phi_{\pm}([T_j]) = [T_j]  \text{   for } j \neq k. \]

Let $X$ be an object of ${\cal{C}}$. For an indecomposable summand $S$ of $T$, 
we denote by $[{\rm{ind}}_{{\cal{T}}}(X):S]$ the coefficient of $S$ in the decomposition of 
$\rm{ind}_{{\cal{T}}}(X)$ with respect to the basis given by the indecomposable objects of ${\cal{T}}$. 
We now establish a generalisation of \cite[Theorem 3]{dehy2008combinatorics} for a stably $2$-CY Frobenius
category $\cal{C}$. The proof goes along the same lines as well. 

\begin{theorem}\label{indtheorem} Let $X$ be a rigid object of $\cal{C}$. \begin{magenta}Let ${\cal{T'}}$ be obtained from ${\cal{T}}$ by mutating at the vertex $k$ as above.\end{magenta} We have
\[ {\rm{ind}}_{\cal{T'}}(X) = \begin{cases} \phi_{+}({\rm{ind}}_{\cal{T}}(X)) \text{  if  } [{\rm{ind}}_{\cal{T}}(X):T_k] \geq 0 ;\\
\phi_{-}({\rm{ind}}_{\cal{T}}(X)) \text{  if  } [{\rm{ind}}_{\cal{T}}(X):T_k] \leq 0. 
\end{cases}\] 
\end{theorem}

\begin{proof}
Let us consider a conflation 
\[
\begin{tikzcd}
	{U_1} & {U_0} & X
	\arrow["g", tail, from=1-1, to=1-2]
	\arrow["h", two heads, from=1-2, to=1-3]
\end{tikzcd}
\]
with $U_i \in \cal{T}$. If $T_k$ is neither a direct factor of $U_0$ nor a direct factor of $U_1$, 
then $U_i \in {\cal{T}}\cap{\cal{T'}}$ for $i = 0,1$. In this case, we have 
\[ {\rm{ind}}_{{\cal{T}}}(X) = {\rm{ind}}_{{\cal{T'}}}(X).  \]

We now consider the case where $[\mathrm{ind}_{{\cal{T}}}(X):T_k] = i$ for a positive integer $i$.
 Then $T_k$ occurs $i$ times in $U_0$ and does not occur in $U_1$. Let $U_0 = U'_0 \oplus T_k^i$, where 
 $U'_0$ does not have $T_k$ as a direct factor. The conflation 
\[ 
 \begin{tikzcd}
	{T_k^*} & E & {T_k}
	\arrow[tail, from=1-1, to=1-2]
	\arrow[two heads, from=1-2, to=1-3]
\end{tikzcd} 
\]
gives rise to the following composition of deflations
\[ U_0'\oplus E^i  \longrightarrow U_0' \oplus T_k^i  \longrightarrow X, \]
which in turn gives rise to the diagram
\[\begin{tikzcd}
	{T_k^{*i}} & {U_1'} & {U_1} \\
	{T_k^{*i}} & {U_0'\oplus E^i} & {U_0'\oplus T_k^i = U_0} \\
	& X & X
	\arrow[tail, from=1-1, to=1-2]
	\arrow[two heads, from=1-2, to=1-3]
	\arrow[Rightarrow, no head, from=1-1, to=2-1]
	\arrow[tail, from=1-2, to=2-2]
	\arrow[two heads, from=2-2, to=3-2]
	\arrow[tail, from=2-1, to=2-2]
	\arrow[two heads, from=2-2, to=2-3]
	\arrow[tail, from=1-3, to=2-3]
	\arrow[two heads, from=2-3, to=3-3]
	\arrow[Rightarrow, no head, from=3-2, to=3-3]
\end{tikzcd}\]
whose first two rows and last two columns are conflations. Since $T_k$ is not a summand of $U_1$, we have
 \[{\rm{Ext}}_{\mathcal{C}}^1(U_1, T_k^*) = 0.\] Therefore, the first row splits and we have 
\[ U_1' = T_k^{*i} \oplus U_1. \]
Thus, from the second column rewritten as 
\[ 
 \begin{tikzcd}
	{T_k^{*i}\oplus U_1} & U'_0 \oplus E^i  & {X}
	\arrow[tail, from=1-1, to=1-2]
	\arrow[two heads, from=1-2, to=1-3]
\end{tikzcd} 
\]
we get
\[ {\rm{ind}}_{{\cal{T'}}}(X) = [U_0'] + i([E]-[T_k^*])-[U_1] = \phi_{+}(\rm{ind}_{{\cal{T}}}(X)).\]

Let us now consider the case where $[{\rm{ind}}_{{\cal{T}}}(X):T_k] = -i$, where $i$ is \begin{magenta}a\end{magenta} positive integer. 
Then $T_k$ occurs in $U_1$ with multiplicity $i$ and does not occur in 
$U_0$. Let $U_1 = U_1' \oplus T_k^i$, where $U_1'$ does not have $T_k$ as a
direct factor. As in the previous case, the conflation 
\[ 
 \begin{tikzcd}
	{T_k} & E' & {T_k^*}
	\arrow[tail, from=1-1, to=1-2]
	\arrow[two heads, from=1-2, to=1-3]
\end{tikzcd} 
\]
gives rise to the following inflation
\[ 
U_1' \oplus T_k^i \rightarrow U_1' \oplus E'^i.
\]  
We deduce the diagram
\[\begin{tikzcd}
	{U_1'\oplus T_k^i} & {U_0} & X \\
	{U_1'\oplus E'^i} & {U_0'} & X \\
	T_k^{*i} & {T_k^{*i}}
	\arrow[tail, from=1-1, to=1-2]
	\arrow[two heads, from=1-2, to=1-3]
	\arrow[tail, from=1-1, to=2-1]
	\arrow[tail, from=1-2, to=2-2]
	\arrow[two heads, from=2-1, to=3-1]
	\arrow[two heads, from=2-2, to=3-2]
	\arrow[tail, from=2-1, to=2-2]
	\arrow[two heads, from=2-2, to=2-3]
	\arrow[Rightarrow, no head, from=1-3, to=2-3]
	\arrow[Rightarrow, no head, from=3-1, to=3-2]
\end{tikzcd}\]
whose first two rows and columns are conflations.
Since $T_k^*$ does not appear as a summand of $U_0$ we have 
\[{\rm{Ext}}^1_{\mathcal{C}}(T_k^*, U_0) = 0.\] 
Therefore, the object $U'_0$ decomposes as
\[U'_0 = T_k^{*i} \oplus U_0. \]
Therefore, we can rewrite the second row as the conflation
\[ 
 \begin{tikzcd}
	{U'_1 \oplus E'^i } & {U_0 \oplus T_k^{*i} }  & {X}
	\arrow[tail, from=1-1, to=1-2]
	\arrow[two heads, from=1-2, to=1-3]
\end{tikzcd} 
\]
which yields
\[ {\rm{ind}}_{{\cal{T}'}}(X) = [U_0] - [U_1'] + i([T_k^*] - [E_k']) = \phi_{-}(\rm{ind}_{{\cal{T}}}(X)). \]

\end{proof}
%%%%

\subsection{$g$-vectors as indices} \label{ss:g-vectors as indices}
Let $\mathcal{C}$ be a stably $2$-CY Frobenius category.  Let 
\[
\mathcal{T}^0  = {\rm{add}} (T_1^0, T_2^0, \ldots, T_m^0)
\]
be a cluster tilting subcategory of $\mathcal{C}$ such that $T_1^0, T_2^0, \ldots, T_r^0$ are non-projective indecomposable objects and  $T_{r+1}^0, T_{r+2}^0, \ldots, T_m^0$ are projective-injective indecomposable objects.

Let $\mathbb{T}_r$ be the $r$-regular tree with
initial vertex $t_0$ as in section~\ref{ss:from-ice-quivers}. With each vertex $t$ of $\mathbb{T}_r$, we
associate a cluster-tilting subcategory $\ct(t)$ with indecomposable objects $T_j(t)$, $1 \leq j\leq m$,
such that 
\begin{itemize}
\item[a)] We have $\ct(t_0)=\ct^0$ and $T_j(t_0)=T_j^0$ for $1 \leq j \leq m$ and
\item[b)] If $t$ and $t'$ are linked by an edge labeled $i$, then $\ct(t')$ obtained from
$\ct(t)$ by mutation at the indecomposable $T_i(t)$ so that $T_j(t')=T_j(t)$ for $j\neq i$
and $T_i(t')=T_i(t)^*$.
\end{itemize}
Notice that for each vertex $t$ of $\mathbb{T}_r$, the direct sum
\[
T(t) = T_1(t) \oplus \cdots \oplus T_m(t)
\]
is a basic cluster-tilting object of $\cc$ and that we have $T_j(t)=T_j(t_0)$ for all $j>r$ and all 
vertices $t$ of $\mathbb{T}_r$, cf. section~\ref{ss:extended g-vectors}. 
We say a cluster tilting object is {\em reachable} from $T$ if it is isomorphic
to $T(t)$ for some $t$.

\begin{theorem} \label{thm:index vs g-vector}
Let $T(t_2)$ be a cluster-tilting object reachable from $T = T(t_1)$ and $\ct=\add(T)$. Then we have
 \[ 
 \ind_{\ct}(T_j(t_2)) = \sum_i g_{ij}^{t_1}(t_2)[T_i]. 
 \] 
\end{theorem}

\begin{proof} Let $l$ be the length of the path from $t_0$ to $t$. We will induct on $l$. 
For $l=0$, the claim clearly holds. Suppose we have
\[ 
\ind_{\ct}(T_j(t_2)) = \sum_i g_{ij}^{t_1}(t_2)[T_i]. 
 \] 
 and there is an edge labeled $k$ from $t_1$ to $t_1'$. Put $\ct'=\add(T(t_1'))$.
 If we apply \begin{magenta}Theorem\end{magenta}~\ref{indtheorem} to $X=T_j(t)$, we find that
 \[
 \ind_{\ct'}(T_j(t_2)) = \sum_i  \phi_\eps( g_{ij}^{t_1}(t_2) [T_i]) = \sum_i h_{ij}^{t_1'}(t_2) [T_i], 
 \]
 where $\eps$ is the sign of the $k$th row of the matrix $G(t_1, t_2)=(g_{ij}^{t_1}(t_2))$. It follows
 that the matrix $H(t_1',t_2)=(h_{ij}^{t_1'}(t_2))$ is given by
 \[
 H(t_1', t_2) = E_\eps(Q(t_1)) G(t_1, t_2).
 \]
 Therefore, by equation (\ref{gvectorssign}), we have $H(t_1', t_2)=G(t_1', t_2)$ as claimed.
\end{proof}

\begin{theorem} The map $M \mapsto \mathrm{ind}_{\mathcal{T}^0}(M)$ induces a bijection from the isomorphism classes of rigid indecomposables reachable from $\mathcal{T}^0$ onto the set of $g$-vectors.
\end{theorem}

\begin{proof} We have injectivity by Theorem~\ref{thm:index determines rigid object} and surjectivity by the
definition of reachability.
\end{proof}

\subsection{Derived equivalences associated to sequences of mutations}
As in section~\ref{ss:g-vectors as indices}, we denote by $\cc$ a stably $2$-CY
Frobenius category, by
\[
T^0=T^0_1\oplus \cdots \oplus T^0_r \oplus T^0_{r+1}\oplus \cdots \oplus T^0_m
\]
a cluster-tilting object with associated cluster-tilting subcategory
$\ct^0=\add(T^0)$, by $\mathbb{T}_r$ the $r$-regular tree with root $t_0$ and
by $\ct(t)$ the cluster-tilting subcategory associated with a vertex $t$ of $\mathbb{T}_r$.

For a $k$-linear category $\cs$, a 
(right) $\cs$-module is a $k$-linear functor $M: \cs^{op} \to \Mod k$. We write
$\Mod \cs$ for the category of all $\cs$-modules and $\mod\cs$ for the category 
of $\cs$-modules whose values are finite-dimensional vector spaces. 
We define $\cd\cs$ to be the unbounded derived category of $\Mod\cs$.
Its full subcategory of compact objects is the perfect derived category
$\per(\cs)$. We write $\cd^b(\mod\cs)$ for the full subcategory of $\cd\cs$
whose objects are the complexes $M$ with bounded cohomology and
such that $H^p(M)$ is finite-dimensional for each $p\in\Z$.

Let $\proj(\ct^0)$ be the subcategory of $\Mod \ct^0$ formed by the finitely
generated projective modules. The Yoneda functor yields an equivalence
\[
\ct^0 \iso \proj(\ct^0)
\]
taking an object $T'$ to ${T'}^\wedge=\Hom(?, T')$. Whence an induced equivalence
\[
\ch^b(\ct^0) \iso \per(\ct^0),
\]
\begin{magenta} where $\ch^b(\ct^0)$ denotes the bounded homotopy category of the additive category $\ct^0$.
\end{magenta}
By composing its quasi-inverse with the functor $\ch^b(\ct^0) \to \cd^b(\cc)$ induced by
the inclusion $\ct^0 \to \cc$ we obtain a canonical functor
\[
\Psi: \per(\ct^0) \to \cd^b(\cc).
\]
Let $\cf d$ be the full subcategory of $\per(\ct^0)$ whose objects are
the $2$-term complexes $T'^\wedge \to T''^\wedge$, where $T'$ and $T''$ belong
to $\ct^0$ and the differential is $d^\wedge$ for an inflation $d$ of $\cc$ (the
abbreviation $\cf d$ stands for `fundamental domain').
The image under $\Psi$ of such an object is the complex $d: T' \to T''$, which is
clearly quasi-isomorphic to $\cok(d)$. Thus, the functor $\Psi$ induces
a functor $\cf d \to \cc$.

\begin{lemma} The induced functor $\cf d \to \cc$ is an equivalence
of $k$-categories.
\end{lemma}

\begin{proof} Since $\ct^0$ is a cluster-tilting subcategory of $\cc$, we
know that for each object $X$ of $\cc$, there is a conflation
\[
\begin{tikzcd}
	{T'} & T'' & X
	\arrow[tail, from=1-1, to=1-2]
	\arrow[two heads, from=1-2, to=1-3]
\end{tikzcd} 
\]
with $T'$, $T''\in \ct^0$. This yields essential surjectivity. The 
full faithfulness is easily deduced from the rigidity of $\ct^0$.
To check faithfulness, let us take two objects  of $\cf d$ given by
inflations
\begin{tikzcd}
	{T_1'} & {T_1''}
	\arrow["{i_1}", tail, from=1-1, to=1-2]
\end{tikzcd} and \begin{tikzcd}
	{T_2'} & {T_2''}
	\arrow["i_2", tail, from=1-1, to=1-2]
\end{tikzcd} with images $X_1$ and $X_2$ respectively in $\cc$. Suppose the morphism $f=(a,b)$
\[\begin{tikzcd}
	{T_1'} & {T_1''} \\
	{T_2'} & {T_2''}
	\arrow["{i_1}", tail, from=1-1, to=1-2]
	\arrow["a"', from=1-1, to=2-1]
	\arrow["b", from=1-2, to=2-2]
	\arrow["{i_2}", tail, from=2-1, to=2-2]
\end{tikzcd}\] 
goes to the $0$-morphism from $X_1$ to $X_2$. Then construct a diagram
\[\begin{tikzcd}
	{T_1'} & {T_1''} & {X_1} \\
	{T_2'} & {T_2''} & {X_2}
	\arrow["{i_1}", tail, from=1-1, to=1-2]
	\arrow["{p_1}", two heads, from=1-2, to=1-3]
	\arrow["{p_2}", two heads, from=2-2, to=2-3]
	\arrow["{i_2}", tail, from=2-1, to=2-2]
	\arrow["a"', from=1-1, to=2-1]
	\arrow["b"', from=1-2, to=2-2]
	\arrow["{c=0}"', from=1-3, to=2-3]
	\arrow["h"{description}, dashed, from=1-2, to=2-1]
\end{tikzcd}\]
as follows: From the commutativity of the right hand square in the above diagram, we have $p_2b = 0$. Hence, there exists an $h:T_1'' \rightarrow T_2'$, such that $b = i_2h$.  Now, $i_2hi_1 = bi_1 = i_2a$ implies $hi_1 = a$ since $i_2$ is a monomorphism. Thus, the morphism $(a,b)$ is null-homotopic, which
was to be shown.
To check fullness, we consider a diagram
\[\begin{tikzcd}
	{T_1'} & {T_1''} & {X_1} \\
	{T_2'} & {T_2''} & {X_2}
	\arrow["{i_1}", tail, from=1-1, to=1-2]
	\arrow["{p_1}", two heads, from=1-2, to=1-3]
	\arrow["{p_2}", two heads, from=2-2, to=2-3]
	\arrow["{i_2}", tail, from=2-1, to=2-2]
	\arrow["a", dashed, from=1-1, to=2-1]
	\arrow["b"',dashed,  from=1-2, to=2-2]
	\arrow["c"', from=1-3, to=2-3]
\end{tikzcd}\]
with a given morphism $c: X_1 \to X_2$. The second conflation yields the
exact sequence
\[
\cc(T_1'', T_2'') \to \cc(T_1'', X_2) \to \Ext^1_\cc(T_1'', T_2').
\]
Since $\ct^0$ is rigid, the third term vanishes. So the composition $c\circ p_1: T_1'' \to X_2$ lifts to a morphism $b: T_1'' \to T_2''$. Clearly $b$ induces a morphism
$a: T_1' \to T_2'$ and $c$ is the image of the homotopy class of $(a,b)$.
\end{proof}

The following theorem is an adaptation of a result by Nagao \cite{Nagao13},
cf.~also section~7.5 of \cite{keller2012cluster}.

\begin{magenta}
\begin{theorem}  Let $t$ be a vertex of the $r$-regular tree $\mathbb{T}_r$.
There is a derived equivalence
\[
\Phi(t): \cd(\ct(t)) \to \cd(\ct^0)
\]
satisfying the following condition
\begin{itemize}
\item[(H)] the equivalence $\Phi(t)$ takes each object of the heart $\Mod \ct(t)$ of the canonical $t$-structure on 
$\cd(\ct(t))$ to an object whose homology is concentrated in
degrees $-1$ and $0$. 
\end{itemize}
\end{theorem}
\end{magenta}
\begin{proof}
Let
\[\begin{tikzcd}
	{t=t_l} & {t_{l-1}} & \cdots & {t_2} & {t_1} & {t_0}
	\arrow["{k_l}", from=1-1, to=1-2]
	\arrow["{k_{l-1}}", from=1-2, to=1-3]
	\arrow["{k_{3}}", from=1-3, to=1-4]
	\arrow["{k_2}", from=1-4, to=1-5]
	\arrow["{k_1}", from=1-5, to=1-6]
\end{tikzcd}
\]
be the unique path in $\mathbb{T}_r$ linking $t$ to the root $t_0$. 
We abbreviate $\ct(t_j)$ by $\ct^j$. We proceed by induction
on its length $l$. 
If $l=0$, we let $\Phi(t)$ be the identity functor. Now suppose that $l>0$ and
that we have constructed the equivalence
\[
F=\Phi(t_{l-1}): \cd \ct^{l-1} \iso \cd \ct^0
\]
satisfying the above condition (H) on the heart. Let us
abbreviate $i=k_l$. Let $S_i$ be the simple quotient of the indecomposable projective
module $T_i(t_{l-1})^\wedge$. Consider the subcategories
\[
\cf=F^{-1}(\Mod \ct^0) \cap \Mod \ct^{l-1} \quad\mbox{and}\quad
\cg=F^{-1}(\Si\Mod \ct^0) \cap \Mod \ct^{l-1}.
\]
It is well-known and easy to check that $(\cf,\cg)$ is a {\em torsion pair} in
$\Mod \ct^{l-1}$, i.e. we have $\Hom(G,F)=0$ for all $G\in \cg$ and $F\in \cf$
and for each module $M$, there is a short exact sequence
\[
\begin{tikzcd}
0 \ar[r] & M_\cg \ar[r] & M \ar[r] & M^\cf \ar[r] & 0
\end{tikzcd}
\]
with $M_\cg\in \cg$ and $M^\cf \in \cf$, see section~7.6 of \cite{keller2012cluster}.  
It is clear that each simple object of $\Mod \ct^{l-1}$ must lie in $\cg$ or $\cf$. 
In particular, this holds for $S_i$. Thus, either $FS_i$ lies in $\Mod \ct^0$ or
in $\Sigma \Mod \ct^0$. We will use this fact below.

The object $T_i$ fits into the exchange conflations
\[ 
 \begin{tikzcd}
	{T_i^*} & E & {T_i}
	\arrow[tail, from=1-1, to=1-2]
	\arrow[two heads, from=1-2, to=1-3]
\end{tikzcd} \quad\mbox{and}\quad
 \begin{tikzcd}
	{T_i} & E' & {T_i^*}
	\arrow[tail, from=1-1, to=1-2]
	\arrow[two heads, from=1-2, to=1-3]
\end{tikzcd}.
\]
Let us recall from Prop.~4 of \cite{Palu09a} that there are two canonical derived equivalences 
\[ 
\begin{tikzcd}
\Phi_{\pm}:  \cd \ct^l \ar[r] & \cd \ct^{l-1}
\end{tikzcd}
\]
which both send $T_j^\wedge$ to $T_j^\wedge$ for $j\neq i$. The equivalence $\Phi_+$ sends
$(T^*_i)^\wedge$ to the cone over the
morphism
\[
 \begin{tikzcd}
	{T_i^\wedge} & {E'}^\wedge 
	\arrow[from=1-1, to=1-2]\end{tikzcd}
\]
and the equivalence $\Phi_-$ sends $\Sigma(T^*_k)^\wedge$ to the cone over the morphism
\[
\begin{tikzcd}
E^\wedge \ar[r] & T_i^\wedge.
\end{tikzcd}
\]
Let $S_i^*$ be the simple quotient of the $\ct^l$-module $T_i^{*\wedge}$. 
We have the projective resolution
\[
\begin{tikzcd} 
0 \ar[r] & {T_i}^{*\wedge} \ar[r] & E^\wedge \ar[r] & {E'}^\wedge \ar[r] & {T_i^*}^\wedge \ar[r] & S_i^* \ar[r] & 0
\end{tikzcd}
\]
obtained by splicing the exchange conflations.
Using the above description of the images of ${T_i^*}^\wedge$ under $\Phi_{\pm}$, we check
that $\Phi_+(S_i^*) \cong \Si S_i$ and $\Phi_-(S_i^*) \cong \Si^{-1} S_i$.
If $\Phi(t_{l-1})(S_i)$ belongs to $\Mod \ct^0$, we define $\Phi(t)= \Phi(t_{l-1}) \circ \Phi_-$.
If $\Phi(t_{l-1})(S_i)$ belongs to $\Si\Mod \ct^0$, we define $\Phi(t)= \Phi(t_{l-1}) \circ \Phi_+$.
It is easy to see that with this definition, the equivalence
\[
\Phi(t): \cd(\ct(t)) \iso \cd(\ct^0)
\]
satisfies condition (H).

\end{proof}

We recall from \cite{Palu09a} that the functors $\Phi_{\pm}$ also induce equivalences 
\[ 
\per(\ct^{l-1}) \iso \per(\ct^l) \quad\mbox{and} \quad \cd^b(\mod \ct^{l-1}) \iso \cd^b(\mod \ct^l).
\]
Thus, the functor $\Phi(t)$ induces isomorphisms 
\[ 
K_0({\rm{per}} (\ct(t))) \iso K_0({\rm{per}}(\mathcal{T}^0))  \quad\mbox{and}\quad
K_0(\mathcal{D}^b (\rm{mod} \ct(t))) \iso K_0(\mathcal{D}^b (\rm{mod}\mathcal{T}^0)).
\]
For $1 \leq i\leq m$, let $P_i(t)= T_i(t)^\wedge$ and let $S_i(t)$ be its unique simple quotient. 
Then the $[P_i(t)]$ and the  $[S_i(t)]$ form dual bases in $K_0(\per(\ct(t)))$ and 
$K_0(\cd^b(\mod\ct(t)))$ for the pairing induced by 
\[ 
\RHom(,) : \per(\ct(t)) \times \cd^b(\mod\ct(t)) \rightarrow \per(k). 
\]
% 27 April, 2022
Hence, the images under $\Phi(t)$ of the $P_i(t)$ and the $S_i(t)$ yield dual bases of 
\[
K_0({\rm{per}}(\ct^0))\quad\mbox{and}  \quad K_0(\mathcal{D}^b (\rm{mod}\ct^0)). 
\]

\begin{lemma} Let $Q(t_0)=Q$ be the ice quiver of the endomorphism algebra of $T^0$.
\begin{itemize}
\item[a)] The quiver of the endomorphism algebra of $\ct(t)$ is $Q(t)$.
\item[b)] The coordinates of $[\Phi(t)(P_j(t))]$ in the basis of the $[P_i(t_0)]$ are
$g^{t_0}_{ij}(t)$, $1\leq i \leq m$ for $Q$.
\item[c)] \begin{magenta}The coordinates of the class $[FS_j(t)]$ in the basis of the $[S_i(t_0)]$ 
are the $c_{ij}(t)$, where $C(t)$ is the $c$-matrix at $t$.\end{magenta}
\end{itemize}
\end{lemma}

\begin{proof} We prove the claims simultaneously using induction on the distance
between $t_0$ and $t=t_l$ in the regular tree and the second construction of the
extended $g$-vectors given in section~\ref{ss:extended g-vectors}. For
$l=0$, the quiver $Q(t)$ equals $Q$ and the functor $\Phi(t)$ is the identity. So the basis given by the
$[P_i(t)]$ coincides with that of the $[P_i(t_0)]$ and the basis of the $[S_i(t)]$
coincides with that of the $[S_i(t_0)]$. This shows the two claims for $l=0$.
Now suppose the claims hold for $l-1$ and that $t_{l-1}$ is linked to $t=t_l$ by an
edge labeled $i=k_l$. Put $F=\Phi(t_{l-1})$. Since the cluster-tilting subcategories
define a cluster structure on $\cc$, the quiver $Q(t)$ of the endomorphism algebra
of $\ct(t)$ is obtained from the quiver $Q(t_{l-1})$ of the endomorphism algebra
of $\ct(t_{l-1})$ by mutation at $i$. By the induction hypothesis, the class of 
$F(S_i(t_{l-1}))$ is given by $c_i(t_{l-1})$. So the object $FS_i(t_{l-1})$ lies in $\Mod \ct^0$
if and only if $c_i(t_{l-1})$ is positive, i.e. if the vertex $i$ is green and
$FS_i(t_{l-1})$ lies in $\Si \Mod \ct^0$ if and only if $c_i(t_{l-1})$ is positive, i.e. the
vertex $i$ is red. Let us define $\eps=1$ if $i$ is green and $\eps=-1$ if it is red.
We have defined $\Phi(t_l)= \Phi(t_{l-1}) \circ \Phi_\eps$. Thus, the matrix of the
map induced in the Grothendieck group by $\Phi(t_l)$ is the product 
\begin{equation}
G(t_0,t_{l-1})E_{k,\eps}(Q(t)).
\end{equation}
Thanks to equation (\ref{gvectorscolor}), this implies the claim for the
$g_j(t_l)$. By duality, it also implies the claim for the $c_j(t_l)$.
\end{proof}

\begin{remark} The above proof shows that the sign-coherence of the
$c$-vectors and the $g$-vectors is a consequence of the existence
of a $2$-Calabi--Yau realization of the ice quiver $Q$.
\end{remark}

%%%%

\section{Cluster algebra structure on the Grassmannian}\label{s.Introduction 2} 

Let $\grass{k}{n}$ denote the {\em Grassmannian variety} 
consisting of $k$-dimensional subspaces of $\mathbb{C}^n$. 
The {\em Pl\"ucker embedding} \cite[Chapter 4]{lakshmibai2007standard} is the map
\[ \grass{k}{n} \longrightarrow {\mathbb{P}}(\bigwedge^k\mathbb{C}^n)  \]
\[    V \mapsto [v_1 \wedge v_2  \ldots \wedge v_k] \]
where $v_1$, \ldots, $v_k$ is any basis of the subspace $V$ of $\mathbb{C}^n$. 
Let $e_1, \ldots , e_n$ denote the standard basis of $\mathbb{C}^n$. Let $I(k,n)$ be
the set of $k$-tuples
\[
\{(i_1,i_2, \ldots, i_k) | \ 1 \leq i_1 < i_2 < \ldots <i_k \leq n \}.
\]
 Then the wedge products $e_{\tau} = e_{i_1} \wedge e_{i_2} \ldots \wedge e_{i_k}$, ${\tau \in I(k,n)}$, 
 form a basis of $\bigwedge^k\mathbb{C}^n$. The dual basis in $(\bigwedge^k\mathbb{C}^n)^*$ is
 formed by the {\em Pl\"{u}cker coordinates $p_{\phi}$, $\phi\in I(k,n)$.}

Let $\mathbb{C}[\grass{k}{n}]$ be the homogeneous coordinate ring of the Grassmannian for the Pl\"{u}cker embedding. We have
\[ \mathbb{C}[\grass{k}{n}] = \mathbb{C}[p_{\tau} | \tau \in I(k,n)]. \]  

Let $G = SL_n(\mathbb{C})$ denote the group of $n \times n$ matrices with determinant $1$. Let $B$ be the Borel subgroup of $G$ consisting of upper triangular matrices in $G$ and let $T$ be the maximal torus consisting of the diagonal matrices in $G$. Let $X(T)$ denote the group of characters of $T$.  In the root system $R$ of $(G,T)$,
let $R^{+}$ denote the set of positive roots with respect to $B$. Let  $S=\{\alpha_1,\ldots,\alpha_{n-1}\}\subseteq R^{+}$ denote the set of simple roots and $\{\omega_1,\ldots,\omega_{n-1}\}$ the fundamental weights.

Let $N_G(T)$ denote the normalizer of $T$ in $G$. The Weyl group $W$ of $G$ is defined to be the quotient $N_G(T)/T$, and for every $\alpha \in R$ there is a corresponding reflection $s_{\alpha} \in W$.  
The Weyl group $W$ is generated by the simple reflections $s_\alpha$ associated with the simple
roots $\alpha$. This also defines a length function $l$ on $W$.
% and the Bruhat order on $W$.

For a subset $K\subseteq S$, denote by $W^{K}$ the set of elements $w\in W$ such that
$w(\alpha)>0$ for all $\alpha\in K$. Let $W_{K}$ be the subgroup of $W$ generated by 
the $s_{\alpha}$, $\alpha \in K$. We recall that $W^K$ is a system of representatives of
minimal length of the cosets of $W$ modulo $W_K$. In particular, every $w\in W$ can be
uniquely expressed as $w=w^{K}w_{K}$, with $w^{K}\in W^{K}$ and $w_{K}\in W_{K}$. 
For $w\in W$, let $n_w \in N_{G}(T)$ be a representative of $w$.  We denote by $P_{K}$ the parabolic subgroup 
of $G$ generated by $B$ and the $n_{w}$, $w \in W_{K}$. Then $W_{K}$ is the Weyl group of the parabolic subgroup $P_{K}$ and abusing notation we also denote it as $W_{P_{K}}$.
When $K = S \setminus \{\alpha_k\}$, then $P=P_K$ is a maximal parabolic
subgroup and the quotient $G/P$ is canonically isomorphic to $\grass{k}{n}$. Now 
we have $W_P = S_k \times S_{n-k}$,  so the minimal length coset representatives of 
$W/W_P$  can be identified with the elements $w\in W$ such that we have
\[
w(1) < w(2) < \ldots < w(k) \mbox{ and } w(k+1) < w(k+2) < \ldots < w(n).
\]
For $K=S \setminus \{\alpha_k\}$, there is a natural identification of  $W^K$ with $I(k,n)$ sending 
$w \in W^K$ to $(w(1),w(2),\ldots,w(k))$.   For $w$ in $I(k,n)$, let $e_{w}$
be the point 
\begin{equation*}
[e_{w(1)} \wedge e_{w(2)}\wedge \cdots \wedge e_{w(k)}]
\end{equation*}
of ${\mathbb P}(\bigwedge^k\mathbb{C}^n)$. Then $e_w$ is a $T$-fixed point of $\grass{k}{n}$ and
in this way, we obtain all the $T$-fixed points of  $\grass{k}{n}$. 
The $B$-orbit $C_w$ through $e_w$ in $G/P$ is the {\em Schubert cell} and its Zariski closure in 
$G/P$ is the {\em Schubert variety $X(w)$}. The {\em Bruhat order} is the order on the $k$-tuples in $I(k,n)$ given 
by containment of Schubert varieties. In this order, we have $v \leq w$ iff $v(i) \leq w(i)$ for $1 \leq i \leq k$. 

Let $w = (a_1,a_2,\ldots,a_k) \in I(k,n)$. 

\begin{magenta}
\begin{definition}\label{defnYoung} The \emph{(increasing) partition} associated with $w$ is 
 $\bf{w} = (\bf{a_1},\bf{a_2},\ldots,\bf{a_n})$, where ${\bf{a_i}} = a_{k-i+1}- (k-i+1)$. 
The \emph{Young diagram ${\cal{Y}}_{w}$} associated with $w$ is the Young diagram whose 
$i$-th row from the top has $\bf{a_i}$ boxes. 
\end{definition}
\end{magenta}
We recall from \cite{lakshmibai1990multiplicities} that ${\cal{Y}}_{w}$ is also the tableau associated with the Schubert variety $X(w)$. 
\begin{example} \label{ex:Young diagram}
Let $k = 3$ and $n = 7$. Let $w = (3,5,7)$. Then ${\cal{Y}}_{w}$ is
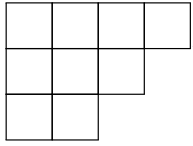
\[
\begin{ytableau}
      \none& & &  & \\
      \none& &  &\\
      \none& & \\
     \end{ytableau}.\] 
\end{example}

The cluster algebra structure on the Grassmannian of planes in complex space $\grass{2}{n}$ was first studied by Fomin--Zelevinsky in \cite{fomin2003cluster}. They showed that it is of cluster type $A_{n-3}$. Using a generalisation of double wiring diagrams due to Postnikov (see \cite{Postnikov06}) Scott constructed a cluster algebra structure on the coordinate ring of $\grass{k}{n}$ in \cite{Scott06}. Later, cluster algebra structures in $\grass{k}{n}$ and other partial flag varieties were studied using categorification by modules over preprojective algebras in \cite{GeissLeclercSchroeer08}. We recall the cluster algebra structure on the coordinate ring of the $\grass{k}{n}$ from \cite{GeissLeclercSchroeer08}. 

\begin{theorem}(\cite[10.3.1]{GeissLeclercSchroeer08}) \label{ini.c}

An initial cluster for $\mathbb{C}[{\grass{k}{n}}]$ consists of the Pl\"{u}cker coordinates $p_{w}$ where $w$ is from the following list 
\begin{eqnarray*}
& \color{blue}{\{2, \ldots k+1\}}, \color{blue}{\{3, \ldots, k+2\}}, {\color{blue}{\ldots},} \color{blue}{\{n-k+2, \ldots n+1\}} \cr
& \{1, 3, \ldots , k+1\}, \{1, 4, \ldots, k+1, k+2\} , \ldots,\color{blue}{ \{1, n-k+3, \ldots n, n+1\}}  \cr
& \cdots , \cdots \cr
& \{1, \ldots ,k-3, k-1, k, k+1\}, \{1, \ldots k-3, k, k+1, k+2\} , \ldots, \color{blue}{\{1, \ldots k-3, n-1, n, n+1\}}  \cr
& \{1, \ldots ,k-2, k-1, k+1\}, \{1, \ldots k-2, k+1, k+2\} , \ldots, \color{blue}{\{1, \ldots k-2, n, n+1\}} \cr
&\{1, \ldots k-1, k+1\}, \{1, \ldots k-1, k+2\} , \ldots, \color{blue}{\{1, \ldots k-1, n+1\}} \cr
&\color{blue}{\{1,2\ldots,k\}}.
\end{eqnarray*}
The $n$ words coloured blue correspond to the frozen cluster variables.
\end{theorem}

We call the seed corresponding to the initial cluster of the theorem the {\em triangular seed}
and write $\boxslash_{k,n}$ for the corresponding quiver. For example, for 
$k= 3$ and $n=7$, this quiver looks like (cf.~section 10.3.1 of \cite{GeissLeclercSchroeer08}, Example~4.3
of \cite{keller2012cluster} or Figure~1 of \cite{Fraser16}).

\[\begin{tikzcd}
	& \color{blue}{234} & \color{blue}{345} & \color{blue}{456} & \color{blue}{567} \\
	& 134 & 145 & 156 & \color{blue}167 \\
	& 124 & 125 & 126 & \color{blue}127 \\
	\color{blue}{123}
	\arrow[from=1-2, to=1-3]
	\arrow[from=1-3, to=1-4]
	\arrow[from=1-4, to=1-5]
	\arrow[from=2-5, to=1-5]
	\arrow[from=3-5, to=2-5]
	\arrow[from=4-1, to=1-2]
	\arrow[from=4-1, to=3-5]
	\arrow[from=2-2, to=1-2]
	\arrow[from=2-3, to=1-3]
	\arrow[from=2-4, to=1-4]
	\arrow[from=3-4, to=2-4]
	\arrow[from=3-3, to=2-3]
	\arrow[from=3-2, to=2-2]
	\arrow[from=2-2, to=2-3]
	\arrow[from=2-3, to=2-4]
	\arrow[from=2-4, to=2-5]
	\arrow[from=3-3, to=3-4]
	\arrow[from=3-4, to=3-5]
	\arrow[from=3-2, to=3-3]
	\arrow[from=1-3, to=2-2]
	\arrow[from=2-3, to=3-2]
	\arrow[from=1-4, to=2-3]
	\arrow[from=1-5, to=2-4]
	\arrow[from=2-4, to=3-3]
	\arrow[from=2-5, to=3-4]
	\arrow[from=4-1, to=3-2]
\end{tikzcd}.
\]

The label $ijk$ on a vertex corresponds to the Pl\"{u}cker coordinate $p_{ijk}$. The vertices colored blue are frozen vertices. The corresponding Pl\"{u}cker coordinate are therefore coefficients and so cannot be mutated.

\begin{remark} The Pl\"ucker coordinate
$p_I$ is an initial cluster variable for the above initial seed if and only if the Schubert variety $X(I)$ is smooth:  
We note that the $(i,j)$-th element from the above array ({\ref{ini.c}})  corresponds to the Pl\"{u}cker coordinate $p_I$ whose associated Young diagram ${\cal{Y}}_I$ is a rectangular tableau with $k-i+1$ rows and $j-i$ columns. 
From \cite{lakshmibai1990multiplicities}, cf. also
{\cite[Corollary 9.3.3]{billey2000singular}}, we know that ${\cal{Y}}_I$ is rectangular 
iff $X(I)$ is smooth. 
\end{remark}

\section{$\mathbf{g}$-vectors for the Pl\"ucker coordinates}\label{s.Introduction 3} 

We use the notations of \begin{magenta}section\end{magenta} \ref{sec:Grassmannian}. For the sequel,
we fix an initial cluster tilting object $T$, namely the sum of the rank-one modules $L_I$,
where $I$ runs through the $k$-subsets described in Theorem~\ref{ini.c}. 
\begin{magenta}Then the 
quiver of the endomorphism algebra of $T$ is also the one from Theorem~\ref{ini.c}.\end{magenta}
The corresponding diagram will be called Jensen--King--Su diagram. 

\begin{definition} 
Let $I \in I(k,n)$. \begin{magenta}Let ${\cal Y}_I$ be the Young diagram as defined in Definition ~\ref{defnYoung}.\end{magenta} We say that a box $b$ of ${\cal Y}_I$ is a {\em peak} if ${\cal Y}_I$ 
contains no boxes to the right or below $b$. Let $(i, j)$ denote the coordinate of $b$. Then $b$ 
is a {\em valley} of ${\cal Y}_I$ if the complement of ${\cal Y}_I$ in the $k \times (n-k)$ rectangle
contains no box to the left or above the box with coordinates $(i+1, j+1)$.
\end{definition}

For $I \in I(k,n)$ we can associate to each ${\cal Y}_I$ the Jensen--King--Su module 
$L_{I}$ as follows.  Let ${\cal Y}_I^\mathrm{T}$ denote the transpose of the Young 
diagram ${\cal Y}_I$. Rotate it by $3\pi/4$ in the counterclockwise direction. 
We identify the upper rim of this rotated diagram with the upper rim of the JKS diagram of the
module $L_I$ associated with $I$, cf.~Figure~\ref{fig:rim}.

\begin{example} \label{example819}

Let $k=8$ and $n=19$. Let $I = (2,3,5,6,7,14,15,19)$. We consider the JKS module $L_I$. 
The module is pictured in Figure~\ref{fig:rim}. The attached Young diagram ${\cal Y}_I$ is 
\bigskip
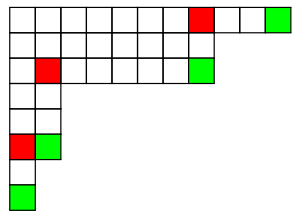
\begin{center}
\ytableausetup{smalltableaux}
\begin{ytableau}
$$ & & & & & & & *(red!)& & & *(green!)\cr
&   & & & & & &   \cr
&   *(red!)& & & & & &*(green!)   \cr
&        \cr
&        \cr
*(red!) &  *(green!) \cr
$$ \cr
*(green!)\cr
\end{ytableau} 
\end{center}
\bigskip

\begin{figure}
\centering
\begin{tikzpicture}[scale=0.7]
\foreach \j in {0,...,19}
  {\path (\j,10.0) node (a) {$\j$};}
\path (0,4) node (a0) {$\circ$}; \path (19,7) node (a19) {$\circ$};  
\foreach \v/\x/\y in
  {a1/1/5, a2/2/4, a3/3/3, a4/4/4, a5/5/3, a6/6/2, a7/7/1, a8/8/2, a9/9/3, a10/10/4, a11/11/5, a12/12/6, a13/13/7, a14/14/6, a15/15/5, a16/16/6, a17/17/7, a18/18/8, b0/0/3, b1/1/4, b2/2/3, b3/3/2, b4/4/3, b5/5/2, b6/6/1, b7/7/0, b8/8/1, b9/9/2, b10/10/3, b11/11/4, b12/12/5, b13/13/6, b14/14/5, b15/15/4, b16/16/5, b17/17/6, b18/18/7,  b19/19/6}
  {\path (\x,\y) node (\v) {$\bullet$};}
\foreach \j in {1,5,9,13,17}
  {\path (\j,-1.2) node {$\vdots$};}
\foreach \t/\h in
  {a1/a2, a2/a3, a4/a5, a5/a6, a6/a7, a13/a14, a14/a15, a18/a19, b1/b2, b2/b3, b4/b5, b5/b6, b6/b7, b13/b14, b14/b15, b18/b19}
  {\path[->,>=latex] (\t) edge[blue,thick] node[black,above right=-2pt] {$x$} (\h);}
\foreach \t/\h in
  {a1/a0, a4/a3, a8/a7, a9/a8, a10/a9, a11/a10, a12/a11, a13/a12, a16/a15, a17/a16, a18/a17, b1/b0, b4/b3, b8/b7, b9/b8, b10/b9, b11/b10, b12/b11, b13/b12, b16/b15, b17/b16, b18/b17}
  {\path[->,>=latex] (\t) edge[blue,thick] node[black,above left =-2pt] {$y$} (\h);}
\end{tikzpicture}
\caption{Example of a module $L_I$}
\label{fig:rim}
\end{figure}
\eat{\begin{figure}[h]	
\centering
\begin{subfigure}[t]{5cm}
\centering
\includegraphics[width=5cm]{jks1}
\caption{}
\label{JKS1}
\end{subfigure}
\caption{\footnotesize JKS module for $L_{\{2,3,5,6,7,14,15,19\}}$}
\end{figure}}
Here the peaks are coloured green while the valleys are coloured $red$.
We also note that the (matrix entry) positions of the peaks are 
$\{(1,11), (3,8),  (6,2), (8,1)\}$ and the positions of the 
valleys are $\{(1,8), (3,2), (6,1)\}$.
\end{example}

Denote the number of peaks by $n_p$ and number of valleys by $n_v$.
\begin{lemma}  \label{lemma:peaks and valleys}
We have $n_p - n_v  = 1$.
\end{lemma}

\begin{proof}
Let $v$ be a valley in position $(i,j)$. Then the bottom-most box in the $j$th column 
and the rightmost box in the $i$th row are both peaks.
\end{proof}

A subset $I$ of the set $\{1, \ldots, n\}$ is a {\em cyclic interval of length $k$} if it is in
the orbit of $\{1, \ldots, k\}$ under the cyclic group action generated by the
permutation mapping $i$ to $i+1$ for $i<n$ and $n$ to $1$.
\begin{lemma} Up to isomorphism, the indecomposable projectives in $\cm(B)$ 
are the modules $L_I$, where $I$ is a cyclic interval of length $k$.
\end{lemma}
\begin{proof}
We know from \cite{JensenKingSu16} that the indecomposable projectives are of
the form $e_jB$ and it is not hard to check that these are the modules $L_I$
associated with cyclic intervals $I$ of length $k$.
\end{proof}
\begin{remark} Let $I\in I(k,n)$. Then the module $L_I$ is projective in $\cm(B)$ if and only if the Young diagram ${\cal Y}_I$ is empty or is a rectangle with $k$ rows or $n-k$ columns.
\end{remark}

\begin{remark} \label{rk:grading}
Let $I\in I(k,n)$. As we have recalled in \begin{magenta}Section\end{magenta}~\ref{sec:Grassmannian},  the module $L_I$ 
admits a $\Ga^\vee$-grading unique up to a multiple of the degree of $t$. We define the $\Ga^\vee$-graded module 
$\tilde{L}_I$ to be $L_I$ endowed with the unique $\Ga^\vee$-grading such that $L_I \cdot e_0$ is generated over 
$Z=\C[[t]]$ in degree~$0$.
\end{remark}

\begin{notation} We write $T_\emptyset$ for the projective $L_I$, where $I=(1,2, \ldots, k)$.
For $1\leq p\leq k$ and $1\leq q\leq n-k$, we write $T_{p,q}$ for the module $L_I$,
where $I$ is determined by the condition that $\cy_I$ rectangular with a unique
peak at $(p,q)$. We write $t_\emptyset$ respectively $t_{p,q}$ for the
canonical generator of $T_\emptyset$ respectively $T_{p,q}$.
\end{notation}

Let $T$ be the direct sum of $T_\emptyset$ and the $T_{p,q}$ for $1\leq p\leq k$ and
$1\leq q\leq n-k$. It follows from Prop.~5.6 and Remark~5.7 of \cite{JensenKingSu16}
that $T$ is a cluster-tilting object in $\cm(B)$. Its associated quiver is the quiver
$\boxslash_{k,n}$ of the triangular seed of Theorem~\ref{ini.c}. We write $m$
for the number of vertices of this quiver so that $g$-vectors with respect
to the triangular seed are elements of $\Z^m$.

\begin{theorem} \label{thm: g-vectors}
Let $I \in I(k,n)$. If $\mathcal{Y}_I$ is non empty, let $P$ denote the set of peaks and $V$ 
denote the set of valleys appearing in ${\cal Y}_I$. Then we have
\begin {itemize}
\item
If $I = (1,2,\ldots, k)$, then the $g$-vector of the Pl\"ucker coordinate $p_I$
with respect to the triangular seed of Theorem~\ref{ini.c}  is the basis vector $e_\emptyset$ of $\Z^m$ 
associated with the exceptional frozen vertex of $\boxslash_{k,n}$.
\item
If $I \neq (1,2,\ldots, k)$, then the $g$-vector of the Pl\"ucker coordinate $p_I$
with respect to the triangular seed of Theorem~\ref{ini.c} is given by
\[
\sum_{p \in P} e_p - \sum_{v \in V} e_v,
\]
where $e_p$ denotes the standard basis vector of $\Z^m$ associated with the vertex $p$ of
the quiver $\boxslash_{k,n}$. 
\end{itemize}
\end{theorem}

\section{Proof of Theorem~\ref{thm: g-vectors}}

\subsection{Reminder on the stable category of Cohen--Macaulay modules}
\label{ss: Reminder on the stable category of Cohen--Macaulay modules}
Let us recall the construction of the Jacobian algebra of a quiver with potential.
Let $Q$ be a finite quiver without loops nor $2$-cycles. Let $\hat{\mathbb{C}Q}$ be the completion of the path algebra $\mathbb{C}Q$ at the ideal generated by the arrows of $Q$. So $\hat{\mathbb{C}Q}$ is a topological algebra and the paths of $Q$ form a topological basis so that the underlying vector space is 
\[ \prod_{p {\text{ }} path} \mathbb{C} p
\]
and the multiplication is induced by the composition of paths. Let $C$ denote the closure of the commutator subspace
$[\hat{\mathbb{C}Q},\hat{\mathbb{C}Q}]$. A {\em potential $W$} on $Q$ is an element of $\hat{\mathbb{C}Q}/C$. The pair 
$(Q,W)$ is called a {\em quiver with potential}. For each arrow $a$ of $Q$, the {\em cyclic derivative with respect to a} is 
the unique continuous $\mathbb{C}$-linear map 
\[ 
\partial_a : \hat{\mathbb{C}Q}/C \longrightarrow \hat{\mathbb{C}Q}
\]
which takes the class of a path $p$ to the sum 
\[ \sum_{p = uav} vu \ko
\]
where $p$ ranges over all paths obtained by concatenations of paths $u,a,v$ where $u$ and $v$ are of length $\geq 0$. The {\em Jacobian algebra $J(Q,W)$} of a quiver with potential $(Q,W)$ is the quotient of the algebra  $\hat{\mathbb{C}Q}$ by the closure of the ideal generated by the cyclic derivatives $\partial_a(W)$, where $a \in Q_1$. 

As in section \ref{sec:Grassmannian}, let us take $\Pi$ to be the completed preprojective algebra of type $\tilde{A}_{n-1}$. 
Let $B$ denote the quotient of $\Pi$ by the $n$ relations $x^k-y^{n-k}$. Let $\cm(B)$ denote the category of finitely generated (maximal) Cohen-Macauley $B$-modules. Let $T$ be the cluster-tilting object $T_{\cal P}$ as explained in section \ref{sec:Grassmannian}, where ${\cal P}$ is the $k$-subset of Theorem~\ref{ini.c}. Let ${\cal{C}}$ denote the stable category of $\cm(B)$. It is Hom-finite and $2$-Calabi--Yau. Let $Q$ be the non frozen part of the quiver of theorem \ref{ini.c}. 
For $k=4$ and $n=9$, the quiver $Q$ thus looks like 
\[ \begin{tikzcd}
	\bullet & \bullet & \bullet & \bullet \\
	\bullet & \bullet & \bullet & \bullet \\
	\bullet & \bullet & \bullet & \bullet
	\arrow[from=1-1, to=1-2]
	\arrow[from=1-2, to=1-3]
	\arrow[from=1-3, to=1-4]
	\arrow[from=2-1, to=1-1]
	\arrow[from=3-1, to=2-1]
	\arrow[from=3-2, to=2-2]
	\arrow[from=3-3, to=2-3]
	\arrow[from=3-4, to=2-4]
	\arrow[from=2-4, to=1-4]
	\arrow[from=2-3, to=1-3]
	\arrow[from=2-2, to=1-2]
	\arrow[from=2-1, to=2-2]
	\arrow[from=3-1, to=3-2]
	\arrow[from=3-2, to=3-3]
	\arrow[from=3-3, to=3-4]
	\arrow[from=2-2, to=2-3]
	\arrow[from=2-3, to=2-4]
	\arrow[from=1-2, to=2-1]
	\arrow[from=1-3, to=2-2]
	\arrow[from=1-4, to=2-3]
	\arrow[from=2-2, to=3-1]
	\arrow[from=2-3, to=3-2]
	\arrow[from=2-4, to=3-3]
\end{tikzcd}.
\]
The potential $W$ is obtained as
\[ W = \sum \begin{tikzcd}[sep=small]
	& \bullet \\
	\bullet && \bullet
	\arrow[from=2-1, to=2-3]
	\arrow[from=2-3, to=1-2]
	\arrow[from=1-2, to=2-1]
\end{tikzcd} - \sum \begin{tikzcd}[sep=small]
	& \bullet \\
	\bullet && \bullet
	\arrow[from=2-1, to=1-2]
	\arrow[from=1-2, to=2-3]
	\arrow[from=2-3, to=2-1]
\end{tikzcd},
\]
where the first sum ranges over the positively oriented $3$-cycles and the second
sum over the negatively oriented $3$-cycles of $Q$.
The Jacobian algebra $J = J(Q,W)$ is isomorphic to the stable 
endomorphism algebra $\underline{\mathrm{End}}_B(T)$, cf.~\cite{JensenKingSu16}. 
The isomorphism $J \iso \underline{\mathrm{End}}_B(T)$ is in fact induced by
a triangle equivalence 
\begin{equation} \label{eq: cluster to stable}
\cc_{B_1 \ten B_2} \iso \cc \ko
\end{equation}
where $B_1 = \mathbb{C} \vec{A}_{n-k-1}$, $B_2 =\mathbb{C} \vec{A}_{k-1}$, the
category $\cc_{B_1 \ten B_2}$ is the (generalized) cluster category of $B_1 \ten B_2$
in the sense of Amiot \cite{Amiot09} and $\vec{A}_{m}$ is the equioriented quiver
of type $A$ with $m$ vertices. Recall that the cluster category $\cc_{B_1 \ten B_2}$ is
defined as the triangulated hull of the orbit category
\[
\cd^b(\mod B_1 \ten B_2)/(S^{-1} \Sigma^2)^\Z,
\]
where $S$ is the Serre functor of $\cd^b(\mod B_1 \ten B_2)$. In particular, we have
a canonical triangle functor 
\begin{equation} \label{eq: derived to cluster}
\cd^b(\mod B_1 \ten B_2) \to \cc_{B_1\ten B_2}
\end{equation}
which takes the free module $B_1 \ten B_2$ to a canonical cluster-tilting object
of $\cc_{B_1\ten B_2}$, which, under the equivalence (\ref{eq: cluster to stable}),
corresponds to the chosen cluster-tilting object $T$ of the stable category $\cc$
of Cohen--Macaulay modules. This yields an algebra morphism
\[
B_1 \ten B_2 \to J(Q,W)
\]
inducing a morphism from the quiver of 
$B_1 \ten B_2$ to the quiver $Q$, namely the inclusion of the subquiver
with the same vertices and whose set of arrows consists of all the
horizontal and vertical arrows of $Q$ as in the following example 
where $k=4$ and $n=9$
\[ \begin{tikzcd}
	\bullet & \bullet & \bullet & \bullet \\
	\bullet & \bullet & \bullet & \bullet \\
	\bullet & \bullet & \bullet & \bullet
	\arrow[from=1-1, to=1-2]
	\arrow[from=1-2, to=1-3]
	\arrow[from=1-3, to=1-4]
	\arrow[from=2-1, to=1-1]
	\arrow[from=3-1, to=2-1]
	\arrow[from=3-2, to=2-2]
	\arrow[from=3-3, to=2-3]
	\arrow[from=3-4, to=2-4]
	\arrow[from=2-4, to=1-4]
	\arrow[from=2-3, to=1-3]
	\arrow[from=2-2, to=1-2]
	\arrow[from=2-1, to=2-2]
	\arrow[from=3-1, to=3-2]
	\arrow[from=3-2, to=3-3]
	\arrow[from=3-3, to=3-4]
	\arrow[from=2-2, to=2-3]
	\arrow[from=2-3, to=2-4]
	%\arrow[from=1-2, to=2-1]
	%\arrow[from=1-3, to=2-2]
	%\arrow[from=1-4, to=2-3]
	%\arrow[from=2-2, to=3-1]
	%\arrow[from=2-3, to=3-2]
	%\arrow[from=2-4, to=3-3]
\end{tikzcd} \quad\quad\xhookrightarrow\quad\quad
\begin{tikzcd}
	\bullet & \bullet & \bullet & \bullet \\
	\bullet & \bullet & \bullet & \bullet \\
	\bullet & \bullet & \bullet & \bullet
	\arrow[from=1-1, to=1-2]
	\arrow[from=1-2, to=1-3]
	\arrow[from=1-3, to=1-4]
	\arrow[from=2-1, to=1-1]
	\arrow[from=3-1, to=2-1]
	\arrow[from=3-2, to=2-2]
	\arrow[from=3-3, to=2-3]
	\arrow[from=3-4, to=2-4]
	\arrow[from=2-4, to=1-4]
	\arrow[from=2-3, to=1-3]
	\arrow[from=2-2, to=1-2]
	\arrow[from=2-1, to=2-2]
	\arrow[from=3-1, to=3-2]
	\arrow[from=3-2, to=3-3]
	\arrow[from=3-3, to=3-4]
	\arrow[from=2-2, to=2-3]
	\arrow[from=2-3, to=2-4]
	\arrow[from=1-2, to=2-1]
	\arrow[from=1-3, to=2-2]
	\arrow[from=1-4, to=2-3]
	\arrow[from=2-2, to=3-1]
	\arrow[from=2-3, to=3-2]
	\arrow[from=2-4, to=3-3]
\end{tikzcd}.
\]
Since the functor (\ref{eq: derived to cluster}) is a triangle functor, the
composition
\[
\mod(B_1 \ten B_2) \to \cd^b(\mod B_1 \ten B_2) \to \cc_{B_1 \ten B_2} \to \cc \ko
\]
henceforth denoted by $\Phi$, 
takes short exact sequence to triangles. Moreover, since this functor takes
$B_1 \ten B_2$ to $T$, it takes projective resolutions
\[
\begin{tikzcd}
0 \arrow{r} & P_1 \arrow{r} & P_0 \arrow{r} & M \arrow{r} & 0
\end{tikzcd}
\]
to triangles
\[
\begin{tikzcd}
T_1 \arrow{r} & T_0 \arrow{r} & \Phi M \arrow{r} &  \Sigma T_1
\end{tikzcd}
\]
with $T_0$ and $T_1$ belonging to $\add(T)$. Thus, {\em if $M$ is
of projective dimension $\leq 1$}, we can read
off the (stable) index of $\Phi M$ with respect to $T$ from a
projective resolution of $M$. We will see that in $\cc$, each
rank one module becomes isomorphic to the image $\Phi M$ of
a $B_1\ten B_2$-module of projective dimension~$\leq 1$. 

\subsection{Proof of Theorem~\ref{thm: g-vectors}} We keep the notations
and assumptions of the preceding section. 
If $I$ is a cyclic interval, then $L_I$ is the indecomposable projective which occurs as the
direct summand $T_{p,q}$ of $T$, where $(p,q)$ is the unique peak of ${\cal{Y}}_I$. 
So $p_I$ has the $g$-vector $e_{p,q}$ as claimed. 

From now on, we suppose
that $I$ is not a cyclic interval. Thanks to Theorem~\ref{thm:index vs g-vector},
it suffices to determine the index of $L_I$ with respect to the cluster-tilting
object $T$ in $\cm(B)$. Let $\pi: \cm(B) \to \ul{\cm}(B)$ be the projection
functor onto the stable category of Cohen--Macaulay modules over $B$.
We will first determine the index of $\pi(L_I)$ with respect
to the cluster-tilting object $\pi(T)$ of $\cc=\ul{\cm}(B)$. This will yield
the `non-frozen' part of the $g$-vector of $L_I$. We will then determine the `frozen part'.

Under the equivalence (\ref{eq: cluster to stable})  between the stable category of Cohen--Macaulay
modules and the cluster category, the object $\pi(L_I)$ in fact corresponds
to the image $\Phi M_I$ of the $k \vec{A}_{k-1} \ten k\vec{A}_{n-k-1}$-module
$M_I$ obtained as the submodule of the projective $P_{k-1} \ten P_{n-k-1}$
generated by the components in degrees $(p,q)$, where $(p,q)$
ranges through the peaks of the Young diagram $\mathcal{Y}_I$ 
satisfying $p<k$ and $q<n-k$ (which means that $T_{p,q}$ is not
projective). Since $B_1\ten B_2$ is a finite-dimensional algebra, the multiplicity
of an indecomposable projective $P_i \ten P_j$ in the $l$th component,
$0\leq l \leq 2$,  of the minimal projective resolution of $M_I$ equals the dimension of
\[
\Ext^{l}_{B_1 \ten B_2}(M_I, S_i \ten S_j).
\] 
Indeed, this follows from the fact that $S_i \ten S_j$ is the head of the indecomposable projective $P_i \ten P_j$.
To compute these extension groups, we use a minimal injective resolution of
$S_i\ten S_j$. For this, we tensor the minimal injective resolution of $S_i$  given by
\[
\begin{tikzcd} 0 & {S_i} & {I_{i}} & {I_{i+1}} & 0, \arrow[from=1-1, to=1-2] \arrow[from=1-2, to=1-3] \arrow["{\alpha_i}", from=1-3, to=1-4] \arrow[from=1-4, to=1-5] \end{tikzcd}
\]
where we put $I_{i+1}=0$ if $i=k-1$,
with the corresponding minimal injective resolution of $S_j$ to obtain the minimal
injective resolution as the total complex of the bicomplex
\[
\begin{tikzcd} && {I_{i}\otimes I_{j+1}} && {I_{i+1}\otimes I_{j+1}} \\ && {} \\ {} && {I_i \otimes I_j} && {I_{i+1}\otimes I_{j}} \arrow["{\alpha_i \otimes 1}", from=1-3, to=1-5] \arrow["{1\otimes\alpha_j}", from=3-3, to=1-3] \arrow["{\alpha_i \otimes 1}", from=3-3, to=3-5] \arrow["{1\otimes\alpha_j}"', from=3-5, to=1-5] \end{tikzcd}
\]
where the lower left corner is in bidegree $(0,0)$. Applying $\Hom(M_I,?)$ to this injective resolution,
we find that the complex $\RHom(M_I, S_i\ten S_j)$ is given by 
the total complex of the bicomplex
\[
\begin{tikzcd} {DM_I(i,j+1)} && {DM_I(i+1,j+1)} \\ \\ {DM_I(i,j)} && {DM_I(i+1,j)} \arrow[from=1-1, to=1-3] \arrow[from=3-1, to=1-1] \arrow[from=3-1, to=3-3] \arrow[from=3-3, to=1-3] \end{tikzcd}
\]
with $DM_I(i,j)$ in bidegree $(0,0)$. This is $k$-dual to the total complex of the bicomplex
\[
\begin{tikzcd} {M_I(i,j+1)} && {M_I(i+1,j+1)} \\ \\ {M_I(i,j)} && {M_I(i+1,j)} \arrow[from=1-1, to=3-1] \arrow[from=1-3, to=1-1] \arrow[from=1-3, to=3-3] \arrow[from=3-3, to=3-1] \end{tikzcd}
\]
with $M_I(i,j)$ in bidegree $(0,0)$ (and each arrow of cohomological degree $1$). Notice that each
of the four vector spaces in this diagram is of dimension at most $1$ and that each of the
four linear maps is injective. In order to compute the homology of the corresponding
total complex, we need to distinguish cases according to the position of $(i,j)$ with
respect to the support of $M_I$ corresponding to the shaded region in the
following picture
\[
\begin{tikzpicture}[scale=0.8]
  % Draw the outer rectangle
  \draw (0,0) rectangle (7,5);

  % Labels for width and height
  \node at (3.5,-0.7) {\(n-k-1\)};
  \node[rotate=90] at (-0.7,2.5) {\(k-1\)};

  % Draw the "support" staircase shape
  \draw[very thick] (0,5) -- (0,3.5) -- (1,3.5) -- (1,2.5) -- (2.5,2.5)
                    -- (2.5,1.5) -- (5,1.5) -- (5,0);

  % Hatch fill the staircase shape
  \begin{scope}
    \clip (0,5) -- (0,3.5) -- (1,3.5) -- (1,2.5) -- (2.5,2.5)
               -- (2.5,1.5) -- (5,1.5) -- (5,0) -- (0,0) -- cycle;
    \foreach \x in {-5, ..., 12}
      \draw[gray] (\x,-1) -- (\x+7,7);
  \end{scope}
  
  % Label inside the staircase region
  \node at (1.5,1) {\(\mathrm{supp}(M_I)\)};
\end{tikzpicture}
\]
Six distinct cases arise depending on whether $(i,j)$ 
\begin{enumerate}
\item does not belong to $\supp(M_I)$,
\item is a peak,
\item is a valley,
\item is an inner vertex of a vertical boundary segment,
\item is an inner vertex of a horizontal boundary segment,
\item is in the support but not on its NE-boundary.
\end{enumerate}
The homology of the total complex does not vanish only if
$(i,j)$ is a peak or a valley. If it is a peak, then homology is
one-dimensional and concentrated in degree $0$; if it is
a valley, then homology is one-dimensional and concentrated
in (cohomological) degree $-1$. It follows that $M_I$ has
a minimal projective resolution of the form
\[
\begin{tikzcd}
0 \arrow{r} & P_1 \arrow{r} & P_0 \arrow{r} & M_{I} \arrow{r} & 0 \ko
\end{tikzcd}
\]
where $P_0$ is the direct sum of the indecomposable projectives
$P_p$, where $p$ runs through the peaks, and $P_1$ is the
direct sum of the indecomposable projectives $P_v$, where
$v$ runs through the valleys. By the remarks at the end of
section~\ref{ss: Reminder on the stable category of Cohen--Macaulay modules},
this implies that the `non-frozen' part of the $g$-vector of $L_I$ is given by
\[
\sum_{p \in P'} e_p - \sum_{v \in V} e_v,
\]
where $P'$ is the set of peaks $(p,q)$ of the Young diagram $\mathcal{Y}_I$ 
satisfying $p<n-k$ and $q<k$. By Remark~\ref{rk:no projective summands}, it follows that 
there is a conflation of $\cm(B)$ of the form
\[
\begin{tikzcd}
0 \arrow{r} & T'' \arrow{r} & T'\oplus P \arrow{r}{[p_1, p_2]} & L_I \arrow{r} & 0 \ko
\end{tikzcd}
\]
where $P$ is projective and $T'$ resp. $T''$ is the direct sum of the  $T_x$ associated with
the peaks $x\in P'$ resp. the valleys $x\in V$. We claim that $p_2: P \to L_I$ is the natural
morphism $p_2': \bigoplus_{x \in P''} T_x \to L_I$, where $P''$ is the set of peaks
$(p,q)$ such that $p=k$ or $q=n-k$. Indeed, it is easy to see that the morphism
\[
[p_1, p_2']: \bigoplus_{x\in P'} T_x \oplus \bigoplus_{x\in P''} T_x \to L_I
\]
is a surjective, radical $\add(T)$-approximation of $L_I$, which shows the claim.

\begin{example} Let $I$ be as in \ref{example819}. The $g$-vector of $p_I$ with respect to the triangle seed is given by 
\[ (e_{I_1} + e_{I_2} + e_{I_3} + e_{I_4}) - (e_{I_5}+ e_{I_6}+e_{I_7}) 
\]
where 
\[
I_1 = (1,2,3,4,5,6,7,19) \quad I_2 = (1,2,3,4,5,14,15,16) 
\]
\[
I_3 = (1,2,5,6,7,8,9,10) \quad I_4 = (2,3,4,5,6,7,8,9)
\]
\[ I_5 = (1,2,3,4,5,6,7,16) \quad I_6 = (1,2,3,4,5,8,9,10)
\]
\[ I_7 = (1,2,4,5,6,7,8,9).
\]

\end{example}
\section{Donaldson--Thomas invariants}
\begin{magenta}
In this section, we present the combinatorial and representation-theoretic framework for the study of
Donaldson–Thomas (DT) transformations and the associated $DT$–$F$–polyno\-mials.
In subsection~\ref{ss:combinatorial-construction}, we recall reddening and maximal green
sequences, the permutation $\sigma$ at the end of a reddening sequence, and the
definition of the $DT$–transformation and its associated $F$-polynomials.  Using
Nagao’s description, we reduce the computation of $DTF_Q$ to that of the $F$-polynomials of
the indecomposable injectives over the Jacobian algebra, and then use a natural grading
to identify their graded submodules with right ideals (3D Young diagrams) in an
explicit poset.  We end the section with a detailed example illustrating our method.
\end{magenta}

\subsection{Combinatorial construction} \label{ss:combinatorial-construction}
Let $Q$ be a finite quiver without loops nor $2$-cycles (and without frozen vertices). We suppose that the set of vertices of $Q$ is
the set of integers $\{1, \ldots\ , r\}$.
As in section~\ref{ss:extended g-vectors}, we can define a vertex $i$ of a quiver $Q'$ obtained from $Q$ by iterated mutation to be {\em green} if the corresponding $c$-vector has all non-negative coefficients, otherwise it is defined to be {\em red}. 
Let $\textbf{k} = (k_1,k_2,\ldots,k_N)$ be a sequence of vertices of $Q$. For $1 \leq s \leq N$, we define $Q(\textbf{k},s)$ to be the
mutated quiver
\[
\mu_{k_{s}}\ldots\mu_{k_{2}} \mu_{k_{1}}(Q)
\]
and for $s =0$,  we define $Q(\textbf{k},s)$ to be the original quiver $Q$. We let $\mu_{\textbf{k}}(Q) = Q(\textbf{k},N)$. If the final quiver  
$\mu_{\textbf{k}}(Q)$ has all its vertices red, we call $\textbf{k}$ a {\em reddening sequence}. 
A sequence $\textbf{k}$ is {\em green} if for each $0 \leq s \leq N-1$, the vertex $k_{s+1}$ is green in the partially mutated quiver  
$Q(\textbf{k},s)$. It is said to be {\em maximal green} if it is green and all the mutable vertices of the final quiver $\mu_{\textbf{k}}(Q)$ are red.  

Let $x = (x_1, x_2, \ldots, x_r)$ be the sequence of the initial cluster variables and $(Q,{x})$ the initial seed associated with $Q$. As in 
section~\ref{ss:from-ice-quivers}, let $\mathbb{T}_r$ be the $r$-regular tree with root $t_0$. We suppose that
$Q$ admits a reddening sequence $\textbf{k}$ of length $N$. 
Let $t'$ be the vertex of $\mathbb{T}_r$ reached from $t_0$ by walking along the edges labeled $k_1$, \ldots,\ $k_N$. Then we know from Prop.~2.10 of \cite{BruestleDupontPerotin14} that there exists a unique permutation $\sigma$ of $Q_0$ such that  we have
\[ 
g_i(t') = -e_{\sigma(i)}
\]
for $1 \leq i \leq r$. 
Let $Q'=Q(t')$ and $x'_i=x_i(t')$, $1 \leq i \leq r$. Then $\sigma^{-1}$ is an isomorphism from the initial quiver $Q$ to the quiver $Q'$. It is known that the following data do not depend on the choice of reddening sequence $\textbf{k}$:
\begin{itemize} 
\item the sequence of cluster variables $x_{\sigma^{-1}(i)}(t')$, $1 \leq i \leq r$,
\item the sequence $F$-polynomials $F_{\sigma^{-1}(i)}(t')$, $1\leq i\leq r$, whose
definition we recall below.
\end{itemize}
The {\em Donaldson--Thomas transformation} is the unique automorphism  
\[ 
DT_Q : {\cal{A}}_{Q} \iso {\cal{A}}_{Q}
\]
taking $x_i$ to $x'_{\sigma^{-1}(i)}$.
We recall from Theorem~4.1 of \cite{Keller17} that the $DT$ transformation (if it exists) does not depend on the choice of the reddening sequence $\textbf{k}$. 

We recall the definition of the {\em $F$-polynomials} from \cite{fomin2007cluster}. Let $B$ denote the antisymmetric matrix associated with the quiver $Q$ as described in section \ref{s.Introduction-1}. Let $y_j = {\displaystyle \prod_{i=1}^r x_i^{b_{ij}}}$, where $j$ ranges from $1$ to $r$. To each $t \in \mathbb{T}_r$ we associate a sequence of polynomials $F(t) = (F_1(t), F_2(t), \ldots, F_r(t)) \in \mathbb{Q}[y_1, y_2, \ldots, y_r]^r$ which is defined recursively via:
\begin{itemize}
\item[(i)] $F(t_0) = (1,1,\ldots,1)$.
\item[(ii)] If there is an edge labeled $k$ between $t$ and $t'$ then
\[ \begin{cases}
F_i(t') = F_i(t) \text{  if  } i \neq k \\
F_k(t)F_k(t') = {\displaystyle \prod_{j' \rightarrow k} y_j \prod_{i\rightarrow k}F_j(t) + \prod_{k \rightarrow j'}y_j \prod_{k \rightarrow i} F_i(t)} \text{  otherwise}. 
\end{cases}
\]
\end{itemize}
As above, if the quiver $Q$ admits a reddening sequence, then we define 
\[ DTF_Q = \sigma(F(t')) \in \mathbb{N}[y_1,y_2, \ldots, y_r]^r.
\]
We call the terms of the sequence $DTF_Q$ the {\em $DTF$-polynomials} of the quiver $Q$
and we write $DTF_i$ for the $i$th term of the sequence, $1 \leq i \leq r$.

\begin{example} \label{ex:combDT}
We work with the quiver $Q = Q(t_0) = A_3:1 \rightarrow 2 \rightarrow 3$. We colour the vertices of the initial seed at $t_0$ green. 
\[\begin{tikzcd}
	\color{green}1 & \color{green}2 & \color{green}3
	\arrow[from=1-1, to=1-2]
	\arrow[from=1-2, to=1-3]
\end{tikzcd}.\]
The corresponding $g$-vectors are $g_i(t_0) = e_i$ for $1 \leq i \leq 3$. 

We consider the sequence $\mu_1\mu_2\mu_1\mu_3\mu_2\mu_1$ of mutations (one can use Keller's mutation app
\cite{QuiverMutation06} by adding a framing) along the vertices $t_0,t_1,\ldots, t_6$ of the tree $\mathbb{T}_r$. We also write the corresponding $g$-vectors.  One observes that this is a reddening sequence. 
\[(\begin{tikzcd}
	\color{green}1 & \color{green}2 & \color{green}3
	\arrow[from=1-1, to=1-2]
	\arrow[from=1-2, to=1-3]
\end{tikzcd}, \{e_1, e_2, e_3\})\]
\[\downarrow{\mu_1}
\]
\[(\begin{tikzcd}
	\color{red}1 & \color{green}2 & \color{green}3
	\arrow[tail reversed, no head, from=1-1, to=1-2]
	\arrow[from=1-2, to=1-3]
\end{tikzcd}, \{-e_1+e_2, e_2,e_3\})
\]
\[\downarrow{\mu_2}
\]
\[(\begin{tikzcd}
	\color{green}1 & \color{red}2 & \color{green}3
	\arrow[from=1-1, to=1-2]
	\arrow[tail reversed, no head, from=1-2, to=1-3]
\end{tikzcd},\{-e_1+e_2,-e_1+e_3,e_3\})
\]
\[\downarrow{\mu_3}
\]
\[(\begin{tikzcd}
	\color{green}1 & \color{green}2 & \color{red}3
	\arrow[from=1-1, to=1-2]
	\arrow[from=1-2, to=1-3]
\end{tikzcd},\{-e_1+e_2,-e_1+e_3,-e_1\})
\]
\[\downarrow{\mu_1}
\]
\[(\begin{tikzcd}
	\color{red}1 & \color{green}2 & \color{red}3
	\arrow[tail reversed, no head, from=1-1, to=1-2]
	\arrow[from=1-2, to=1-3]
\end{tikzcd}, \{-e_2+e_3,-e_1+e_3,-e_1\})
\]
\[\downarrow{\mu_2}
\]
\[(\begin{tikzcd}
	\color{green}1 & \color{red}2 & \color{red}3
	\arrow[from=1-1, to=1-2]
	\arrow[tail reversed, no head, from=1-2, to=1-3]
\end{tikzcd},\{-e_2+e_3,-e_2,-e_1\})
\]
\[\downarrow{\mu_1}
\]
\[(\begin{tikzcd}
	\color{red}1 & \color{red}2 & \color{red}3
	\arrow[tail reversed, no head, from=1-1, to=1-2]
	\arrow[tail reversed, no head, from=1-2, to=1-3]
\end{tikzcd},\{-e_3,-e_2,-e_1\})
\]
So we have $Q' = Q(t_6)$. 
The sequence of $F$-polynomials associated with $t_6$ is $F(t_6)=(F_1, F_2, F_3)$, where
\begin{align*}
F_1 &= 1 + y_3 \\
F_2 &=1 + y_2 + y_2y_3\\ 
F_3 &=1 + y_1 + y_1y_2 + y_1y_2y_3.  
\end{align*}
The list of values of the permutation $\sigma$ is $(3,2,1)$. Therefore, the sequence
$DTF_Q$ is $(F_3, F_2, F_1)$.
Another possible maximal green sequence would be $\mu_1\mu_2\mu_3$. 
Then the permutation $\sigma$ is the identity and we obtain the same sequence $DTF_Q$ of $F$-polynomials. 
\end{example}

\subsection{Computation via representations}
For a (right) module $M$ over the path algebra $kQ$ (or, equivalently, a representation of $Q^{op}$), 
where $k=\mathbb{C}$, and a 
dimension vector $\textbf{e}=(e_1, \ldots\ , e_r)$, we denote by $Gr_{\textbf{e}}(M)$ the 
{\em quiver Grassmannian} of subrepresentations 
$N \subseteq M$ of dimension vector $\textbf{e}$. We consider it as a complex projective
variety. In particular, it is a compact topological subspace of some projective space over $\mathbb{C}$ endowed with
the transcendental topology. We write $\chi(Gr_{\textbf{e}}(M))$ for its Euler characteristic.
By definition \cite{DerksenWeymanZelevinsky10},  the {\em $F$-polynomial of the representation $M$} is
\[
F_{M}(y_1,y_2,\ldots,y_r) = \sum_{\textbf{e}} \chi(Gr_{\textbf{e}}(M)) {\displaystyle \prod_{i=1}^r y_i^{e_i}}. 
\] 
\begin{theorem}[Nagao \cite{Nagao13}] \label{thm:Nagao}
If $Q$ admits a reddening sequence and $W$ is any non degenerate potential
on $Q$, then the $i$th $F$-polynomial in the sequence $DTF_Q$ is given by
\[
DTF_{Q,i}=F_{I_i}
\]
where  $I_i$ is the (right) module over the Jacobian algebra of $(Q,W)$ constructed as
the injective hull of the simple module concentrated at the vertex $i$ of $Q$.
\end{theorem}
We refer to \cite{keller2011cluster} for an explanation of the proof.
In the above example~\ref{ex:combDT}, the vanishing potential is non degenerate and
the injective representations are
\begin{align*}
I_1 &= (k \to k \to k) \\
I_2 &= (0 \to k \to k) \\
I_3 &= (0 \to 0 \to k).
\end{align*}
Clearly, the associated $F$-polynomials are exactly those of the sequence $DTF_Q$ computed above.

Our aim is to study the $DTF$-polynomials associated with the Grassmannian
cluster algebra using the representation-theoretic approach. We will show that the coefficients
appearing in these $F$-polynomials equal $0$ or $1$. This result also follows from Weng's description of 
$DTF$-polynomials in \cite{Weng23}, where he studies more generally $DTF$-polynomials  on 
quivers associated with triples of flags. However, we believe we have a simpler 
approach to the same. 

We keep the assumptions and notations of 
section~\ref{ss: Reminder on the stable category of Cohen--Macaulay modules}.
Since the Jacobian algebra $J = J(Q,W)$ is isomorphic to the stable 
endomorphism algebra $\underline{\mathrm{End}}_B(T)$,
the algebra $J$ acts on $\underline{\mathrm{Hom}}(T, M)$ for each module $M$ in $\cal{C}$. 
The $F$-polynomial associated with an object $M \in {\cal{C}}$ is given by 
\[ 
F_M(y) = \sum_{\textbf{e}} \chi(Gr_{\textbf{e}}(\Ext^1(T,M))). {\displaystyle \prod_{i=1}^r y_i^{e_i}}. 
\]
 We observe that when $M \in \add(T)$, then $F_M = 1$ since $\Ext^1(T,T) = 0$. 
 When  $M = \Sigma T_i$, then we have
 \[
 \Ext^1(T,\Sigma T_i)= \Ext^2(T, T_i) = D \underline{\Hom}(T_i, T)=I_i
 \]
 by the $2$-CY property of the stable category. Thus, the $DTF$-polynomials
 of the Grassmannian cluster algebra are the polynomials $F_{I_i}$, $1 \leq i\leq r$,
 where $r=(k-1)(n-k-1)$ is the number of vertices of $Q$.
 
 \subsection{Grading on the quiver} We define a grading $Q_1 \to \Z$ on $Q$ by declaring
 the degree of all diagonal arrows to be $1$ and the degrees of the horizontal and
 the vertical arrows to be $0$. This yields a non negative grading on the path algebra
 $k Q$. Notice that the potential $W$ is homogeneous of degree $1$ for this
 grading. Thus, each cyclic derivative $\partial_\alpha W$ is homogeneous of
 degree $1-|\alpha|$. Therefore, the ideal defining the Jacobian algebra is homogeneous
 and the Jacobian algebra inherits a grading from the path algebra. Clearly,
 each indecomposable injective module $I_i$ is naturally graded.
 
 Let $M$ be a finite-dimensional graded $J$-module. For a dimension
 vector $\mathbf{e}$, we denote by $Gr^\Z_{\mathbf{e}}(M)$ the
 Grassmannian of {\em graded} submodules $N \subseteq M$ of
 dimension vector $\mathbf{e}$. We have a $k^\times$-action on
 $M$ given by
 \[
 t.m = t^{|m|}m
 \]
 for homogeneous elements $m\in M$. It induces a $k^\times$-action
 on the quiver Grassmannian $Gr_{\mathbf{e}}(M)$ and the
 subvariety $Gr^\Z_{\mathbf{e}}(M)$ is the fixed point set of
 this action. By Bialynicki-Birula's theorem \cite{BialynickiBirula73},
 we have the equality of Euler characteristics
 \[
 \chi(Gr^\Z_{\mathbf{e}}(M)) = \chi(Gr_{\mathbf{e}}(M)).
 \]
 In particular, we may compute the $i$th $DTF$-polynomial using the
 formula
 \begin{equation} \label{eq:CChomog}
 F_{I_i}(y)=  \sum_{\textbf{e}} \chi(Gr^\Z_{\textbf{e}}(I_i)). {\displaystyle \prod_{i=1}^r y_i^{e_i}}.
 \end{equation}
 In order to describe the structure of the graded submodules of $I_i$, we need
 the following lemma. 
 We identify the set $Q_0$ of vertices of $Q$ with the 
 product set
 \[
 \{1, \ldots\ , n-k-1\} \times \{1, \ldots\ , k-1\}.
 \]
 Let us denote all horizontal arrows of $Q$ by $a$,
 all vertical arrows by $b$ and all diagonal arrows by $c$.
 For two integers $u\leq v$, we denote by 
 $\llbracket u,v \rrbracket$ the integer interval
 $\{u, u+1, \ldots, v\}$.
 
 \begin{lemma} \label{keylemmadt} Let $i=(p,q)$ be a vertex of $Q$. For an
 integer $d$ and a vertex $j=(p',q')$ of $Q$, the homogeneous component of degree $d$
 of $e_j J e_i$ is at most one-dimensional. It does not
 vanish  if and only if we have
 \[
 0 \leq d \leq \min(p-1,q-1)
 \] 
 and $j$ belongs to the image of the rectangle
 \begin{equation} \label{eq:R_i}
 R_i=\llbracket p, n-k-1\rrbracket \times \llbracket q, k-1\rrbracket
 \end{equation}
 under the translation by the vector $(-d, -d)$.
 In this case, the homogeneous component is generated by
 the image in $J$ of the path $a^r b^s c^d$, where 
$r=p'-p+d$, and $s=q'-q+d$.
\end{lemma}

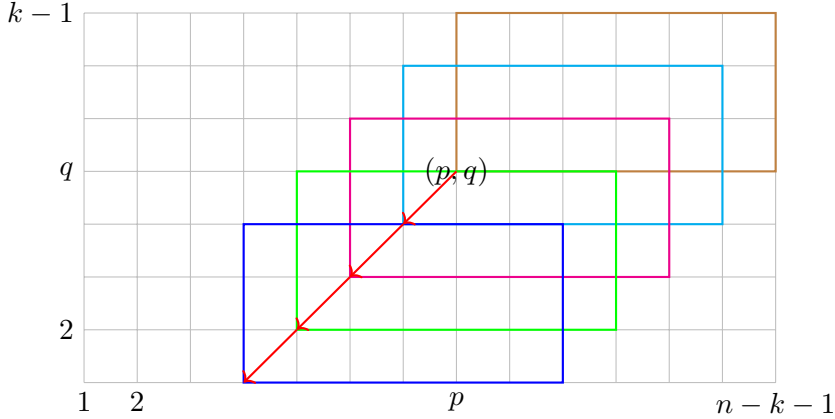
\begin{tikzpicture}[scale=0.70]
    % Draw the grid
    \draw[gray!50, thin] (0,0) grid (13,7);
    % Draw the rectangles
   
    %\draw[black, thick] (0,0) rectangle (13,7);
    \draw[brown,thick] (7,4) rectangle(13,7);
    \draw[cyan,thick] (6,3) rectangle(12,6);
    \draw[magenta, thick](5,2) rectangle(11,5);
    \draw[green, thick](4,1) rectangle(10,4);
    \draw[blue, thick](3,0) rectangle(9,3);
    
       % Draw the arrows
    \draw[red, thick, ->] (7,4) -- (6,3);
    \draw[red, thick, ->] (6,3) -- (5,2);
    \draw[red, thick, ->] (5,2) -- (4,1);
    \draw[red, thick, ->] (4,1) -- (3,0);
    
    \node at (7,0) [below] {$p$};
    \node at (0,4) [left] {$q$};
    \node at (0,0) [below]{$1$};
    \node at (1,0) [below]{$2$};
     \node at (0,1) [left]{$2$};
  \node at (13,0)[below]{$n-k-1$};
  \node at (0,7)[left]{$k-1$};
  \node at (7,4) {$(p,q)$};
     \end{tikzpicture}

\begin{proof}  Let $\tilde{Q}$ be the infinite graded quiver with vertex set
$\Z \times \Z$ and with arrows
\[
a: (p,q) \to (p+1,q) \ko
b: (p,q) \to (p, q+1) \ko
c: (p,q) \to (p-1, q-1)
\]
for all $(p,q)\in \tilde{Q}$, where $a$ and $b$ are of degree $0$ and
$c$ is of degree $1$.  Let $\tilde{\cj}$ be the quotient of the
graded path category of $\tilde{Q}$ by all commutativity relations
\begin{equation} \label{commutativity}
ab=ba\ko bc = cb \ko ac = ca.
\end{equation}
Since these relations are homogeneous, $\tilde{\cj}$ is
naturally graded.  Let $\cj$ be the quotient of $\tilde{\cj}$ by the ideal
generated by the identities of all objects $(p,q)$ not lying
in the rectangle
\[
 R_+=\llbracket 1, n-k-1\rrbracket \times \llbracket 1, k-1\rrbracket.
 \]
 Clearly, the algebra $J$ is the graded `matrix algebra'
 \[
 J = \bigoplus \cj(i,j)
 \]
where $i$ and $j$ range over all vertices of $R_+$. We need to
show that each graded component $\cj(i,j)_t$, $t\in \Z$, is of
dimension at most $1$. In view of the relations~(\ref{commutativity}),
each morphism from $i=(p,q)$ to $j=(p',q')$ in the category
$\tilde{\cj}$ can be uniquely written in the form $a^r b^s c^t$,
where $(p'-p, q'-q)=(r-t, s-t)$. In particular, the component
$\tilde{\cj}(i,j)_t$ is of dimension at most one for each
fixed $t\in \Z$. Clearly, the quotient category $\cj$ inherits
this property. It remains to be determined which morphisms
$a^r b^s c^t$ have non zero images in $\cj$. Clearly,
such a morphism has vanishing image in $\cj$ if and only
if it factors through a vertex outside the rectangle $R_+$.
Suppose that $\alpha=a^r b^s c^t$  is a morphism from $i=(p,q)$
to $j=(p',q')$ in $\tilde{\cj}$. If $r>(n-k-1)-p$, then $\alpha=b^s c^t a^r$
factors through $(p+r, q)$, which lies outside of $R_+$, and if
$s>(k-1)-q$, then $\alpha=a^r c^t b^s$ factors through
$(p, q+s)$, which lies outside of $R_+$. On the other
hand, the morphisms $a^r b^s: (p,q) \to (p+r,q+s)$ for
$r\leq (n-k-1)-p$ and $s\leq (k-1)-q$ clearly have
non vanishing images in $\cj$ and so do their
compositions with $c^t$ for $t \leq \min(p-1, q-1)$.
\end{proof}

Let us use the lemma to elaborate on the structure of the
injective indecomposable module $I_i$ associated with
a vertex $i=(p,q)$ of $Q$. Recall that the value at a vertex
$j=(p',q')$ of the representation of $Q^{op}$ associated
with $I_i$ is the dual 
\[
\Hom_k(e_i J e_j, k) = \Hom_k(\cj(i,j), k) \ko
\]
where we have used the notation $\cj$ from the proof of the
lemma. Let $R_i$ denote the rectangle (\ref{eq:R_i}) of the lemma.
It follows from the lemma, that 
\begin{itemize}
\item[a)] $I_i$ is the direct sum of its homogeneous components $(I_i)_d$ for 
\[
-\min(p-1, q-1) \leq d \leq 0 \ko
\] 
\item[b)] the homogeneous component $(I_i)_d$ is thin and supported
in the rectangle $R_i  + (d,d)$, i.e. the value $(I_i)_d(r,s)$ of $(I_i)_d$
at a vertex $(r,s)$ is of dimension at most $1$ and of dimension one
precisely when $(r,s)$ lies in $R_i+(d,d)$,
\item[c)] right multiplication by an arrow $c$ with target $(r,s)$ in
$R_i+(d-1,d-1)$ induces a bijection $(I_i)_{d-1}(r,s) \iso (I_i)_d(r+1, s+1)$
for each $d$ such that $-\min(p-1, q-1) <d \leq 0$. Similarly for
right multiplication by $a$ and by $b$.
\end{itemize}
It follows that the homogeneous submodules of $I_i$ used in formula
(\ref{eq:CChomog}) are in bijection
with the right ideals (=predecessor closed subsets) of the poset
$L_r \times L_s \times L_t$, where $L_r$ is the linearly ordered set
$1 < 2< \cdots < r$, and $r=(n-k-1)-p$, $s=(k-1)-q$ and
$t=1 + \min(p-1, q-1)$. More precisely, if $e$ is a dimension vector
for $Q$, then the quiver Grassmannian $Gr^\Z_{\textbf{e}}(I_i)$ is
a finite set of points, one for each homogeneous dimension vector
$\tilde{e}$ with associated non homogeneous dimension vector $e$.
If $K$ is a right ideal of $L_r \times L_s \times L_t$, its contribution
to $F_{I_i}$ is the sum of all monomials 
\[
\prod_{(p',q',r')\in K} y_{p+p'-r', q+q'-r'}.
\]
Thus, we have proved the
\begin{theorem}[Weng \cite{Weng23}] \label{thm:Weng} 
For a vertex $i=(p,q)$ of $Q$, the corresponding
$DTF$-polynomial is
\[
F_{I_i}(y)=\sum_K \prod_{(p',q',r')\in K} y_{p+p'-r', q+q'-r'}
\]
where $K$ ranges over the right ideals of the poset 
$L_r \times L_s \times L_t$ with $r=(n-k-1)-p$, $s=(k-1)-q$ and
$t=1 + \min(p-1, q-1)$.
\end{theorem}

\begin{remark} Notice that the right ideals of the poset $L_r \times L_s \times L_t$
may be viewed as the 3D Young diagrams contained in an integral cuboid
of side lengths $r$, $s$ and $t$. This is the viewpoint of Weng in 
Theorem~7.7 of \cite{Weng23}. His theorem is formulated for
the varieties of triples of flags but it is equivalent to the corresponding
formula for the Grassmannians thanks to Proposition~4.3 of
\cite{Weng23}, cf.~also \cite{Muller16, KellerDemonet20}. 
Indeed, by deleting vertices (and the arrows incident with them), we can pass from
the rectangular quiver of a Grassmannian to the triangular
quiver of a variety of triples of flags and vice versa. Notice
that our proof is completely different from Weng's. In particular,
it offers a natural interpretation of the poset of 
right ideals in $L_r \times L_s \times L_t$ as the
poset of graded submodules in an indecomposable
injective module over the Jacobian algebra.
\end{remark}

\begin{proposition} The non zero coefficients of the polynomial
$F_{I_i}$ are equal to $1$.
\end{proposition}

\begin{proof} We need to show that if the quiver Grassmannian
$Gr^\Z_{\textbf{e}}(I_i)$ is non-empty, there is at most one
homogeneous dimension vector $\tilde{e}$ whose associated
non homogeneous dimension vector is $e$. For this, we 
fix an integer $p_0 \in \{1, \ldots, n-k-1\}$ and consider
the function taking $t$ to $e(p_0+t, t)$ when the latter
is defined and to $0$ otherwise.
Since right multiplication with $c$ defines an injection
$I_i(p_0+t-1, t-1) \to I_i(p_0+t, t)$ (whenever both are
defined), this function is increasing and we have
\[
\tilde{e}(p_0+t, t) = e(p_0+t, t) - e(p_0 +t-1, t-1).
\]
\end{proof}

\begin{example} \label{example Gr(4,9)} We work with the example of $Gr(4,9)$. We consider the quiver without the frozen vertices. The indices corresponding to the Pl\"ucker coordinates are marked in the boxes. The initial seed looks as follows: 
\begin{equation} \label{eq:initial-seed}
\begin{tikzcd} 
	\color{green} [1345] & \color{green} [1456] & \color{green} [1 5 6 7] & \color{green} [1 6 7 8 ] \\
	\color{green} [1 2 4 5] & \color{green} [1 2 5 6] & \color{green} [1 2 6 7] & \color{green} [1 2 7 8] \\
	\color{green} [1 2 3 5] & \color{green} [1 2 3 6] & \color{green} [1 2 3 7] & \color{green} [1 2 3 8]
	\arrow[from=1-1, to=1-2]
	\arrow[from=1-2, to=1-3]
	\arrow[from=1-2, to=2-1]
	\arrow[from=1-3, to=1-4]
	\arrow[from=1-3, to=2-2]
	\arrow[from=1-4, to=2-3]
	\arrow[from=2-1, to=1-1]
	\arrow[from=2-1, to=2-2]
	\arrow[from=2-2, to=1-2]
	\arrow[from=2-2, to=2-3]
	\arrow[from=2-2, to=3-1]
	\arrow[from=2-3, to=1-3]
	\arrow[from=2-3, to=2-4]
	\arrow[from=2-3, to=3-2]
	\arrow[from=2-4, to=1-4]
	\arrow[from=2-4, to=3-3]
	\arrow[from=3-1, to=2-1]
	\arrow[from=3-1, to=3-2]
	\arrow[from=3-2, to=2-2]
	\arrow[from=3-2, to=3-3]
	\arrow[from=3-3, to=2-3]
	\arrow[from=3-3, to=3-4]
	\arrow[from=3-4, to=2-4]
\end{tikzcd}
\end{equation}
We use the following maximal green (hence reddening) sequence: Initially, all vertices are 
colored green. We successively mutate the vertices in rows $1$ to $3$ starting at the left vertex of each 
row. After these $12$ mutations, the vertices in the rightmost column have turned red and all others 
are green.
\[\begin{tikzcd} 
	\color{green} [2456] & \color{green} [2 5 6 7] & \color{green} [2 6 7 8] & \color{red} [2 7 8 9] \\
	\color{green} [2 3 5 6] & \color{green} [2 3 6 7] & \color{green} [2 3 7 8] & \color{red} [2 3 8 9]\\
	\color{green} [2 3 4 6] & \color{green} [2 3 4 7] & \color{green} [2 3 4 8] & \color{red} [2 3 4 9]
	\arrow[from=1-1, to=1-2]
	\arrow[from=1-2, to=1-3]
	\arrow[from=1-2, to=2-1]
	\arrow[from=1-3, to=1-4]
	\arrow[from=1-3, to=2-2]
	\arrow[from=1-4, to=2-3]
	\arrow[from=2-1, to=1-1]
	\arrow[from=2-1, to=2-2]
	\arrow[from=2-2, to=1-2]
	\arrow[from=2-2, to=2-3]
	\arrow[from=2-2, to=3-1]
	\arrow[from=2-3, to=1-3]
	\arrow[from=2-3, to=2-4]
	\arrow[from=2-3, to=3-2]
	\arrow[from=2-4, to=1-4]
	\arrow[from=2-4, to=3-3]
	\arrow[from=3-1, to=2-1]
	\arrow[from=3-1, to=3-2]
	\arrow[from=3-2, to=2-2]
	\arrow[from=3-2, to=3-3]
	\arrow[from=3-3, to=2-3]
	\arrow[from=3-3, to=3-4]
	\arrow[from=3-4, to=2-4]
\end{tikzcd}\]     
Now we similarly proceed with the vertices in the green $3 \times 3$-square. After these
$9$ mutations, the last two columns have turned red and the first two are still green. 
\[\begin{tikzcd}
	\color{green} [3 5 6 7] & \color{green}[3 6 7 8] & \color{red}[3 7 8 9] & \color{red}[2 7 8 9] \\
	\color{green}[3 4 6 7] & \color{green}[3 4 7 8] & \color{red}[3 4 8 9] & \color{red}[2 3 8 9] \\
	\color{green}[3 4 5 7] & \color{green}[3 4 5 8] & \color{red}[3 4 5 9] & \color{red}[2 3 4 9]
	\arrow[from=1-1, to=1-2]
	\arrow[from=1-2, to=1-3]
	\arrow[from=1-2, to=2-1]
	\arrow[from=1-3, to=2-2]
	\arrow[from=1-3, to=2-4]
	\arrow[from=1-4, to=1-3]
	\arrow[from=2-1, to=1-1]
	\arrow[from=2-1, to=2-2]
	\arrow[from=2-2, to=1-2]
	\arrow[from=2-2, to=2-3]
	\arrow[from=2-2, to=3-1]
	\arrow[from=2-3, to=1-3]
	\arrow[from=2-3, to=3-2]
	\arrow[from=2-3, to=3-4]
	\arrow[from=2-4, to=1-4]
	\arrow[from=2-4, to=2-3]
	\arrow[from=3-1, to=2-1]
	\arrow[from=3-1, to=3-2]
	\arrow[from=3-2, to=2-2]
	\arrow[from=3-2, to=3-3]
	\arrow[from=3-3, to=2-3]
	\arrow[from=3-4, to=2-4]
	\arrow[from=3-4, to=3-3]
\end{tikzcd}\]
We repeat the same process two more times. The final quiver looks as follows,
which confirms that we have a maximal green sequence (which can also be
checked using the mutation applet  \cite{QuiverMutation06}). The permutation
$\sigma$ (cf.~section~\ref{ss:combinatorial-construction}) associated with this reddening sequence is
the reflection at the vertical central axis. Notice that $\sigma$ does define
a quiver isomorphism between (\ref{eq:final-seed}) and (\ref{eq:initial-seed}).
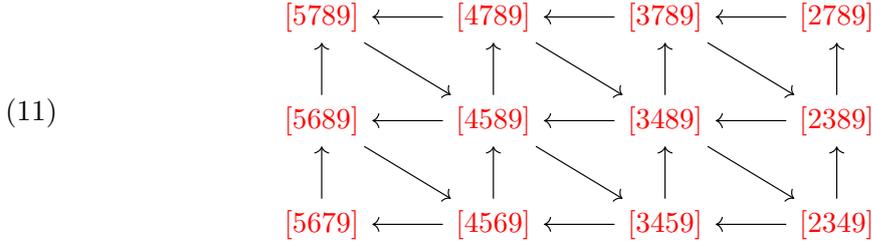
 \begin{equation} \label{eq:final-seed}
 \begin{tikzcd}
	\color{red}[5789] & \color{red}[4789] & \color{red}[3 7 8 9] & \color{red}[2789] \\
	\color{red}[5689] & \color{red}[4589] & \color{red}[3489] & \color{red}[2389] \\
	\color{red}[5679] & \color{red}[4569] & \color{red}[3459] & \color{red}[2349]
	\arrow[from=1-1, to=2-2]
	\arrow[from=1-2, to=1-1]
	\arrow[from=1-2, to=2-3]
	\arrow[from=1-3, to=1-2]
	\arrow[from=1-3, to=2-4]
	\arrow[from=1-4, to=1-3]
	\arrow[from=2-1, to=1-1]
	\arrow[from=2-1, to=3-2]
	\arrow[from=2-2, to=1-2]
	\arrow[from=2-2, to=2-1]
	\arrow[from=2-2, to=3-3]
	\arrow[from=2-3, to=1-3]
	\arrow[from=2-3, to=2-2]
	\arrow[from=2-3, to=3-4]
	\arrow[from=2-4, to=1-4]
	\arrow[from=2-4, to=2-3]
	\arrow[from=3-1, to=2-1]
	\arrow[from=3-2, to=2-2]
	\arrow[from=3-2, to=3-1]
	\arrow[from=3-3, to=2-3]
	\arrow[from=3-3, to=3-2]
	\arrow[from=3-4, to=2-4]
	\arrow[from=3-4, to=3-3]
\end{tikzcd}
\end{equation}
To ease the computation of the $F$-polynomials, let us relabel the vertices of the initial quiver from 
$1$ to $12$ with the vertex in row $i$ and column $j$ assigned the label $4(i-1)+j$.
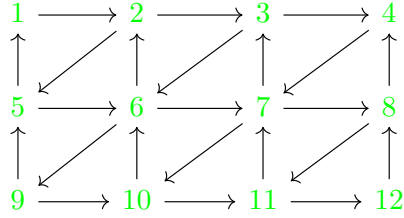
\[\begin{tikzcd} 
	\color{green} 1 & \color{green} 2 & \color{green} 3 & \color{green} 4 \\
	\color{green} 5 & \color{green} 6 & \color{green} 7 & \color{green} 8 \\
	\color{green} 9 & \color{green} 10 & \color{green} 11 & \color{green} 12
	\arrow[from=1-1, to=1-2]
	\arrow[from=1-2, to=1-3]
	\arrow[from=1-2, to=2-1]
	\arrow[from=1-3, to=1-4]
	\arrow[from=1-3, to=2-2]
	\arrow[from=1-4, to=2-3]
	\arrow[from=2-1, to=1-1]
	\arrow[from=2-1, to=2-2]
	\arrow[from=2-2, to=1-2]
	\arrow[from=2-2, to=2-3]
	\arrow[from=2-2, to=3-1]
	\arrow[from=2-3, to=1-3]
	\arrow[from=2-3, to=2-4]
	\arrow[from=2-3, to=3-2]
	\arrow[from=2-4, to=1-4]
	\arrow[from=2-4, to=3-3]
	\arrow[from=3-1, to=2-1]
	\arrow[from=3-1, to=3-2]
	\arrow[from=3-2, to=2-2]
	\arrow[from=3-2, to=3-3]
	\arrow[from=3-3, to=2-3]
	\arrow[from=3-3, to=3-4]
	\arrow[from=3-4, to=2-4]
\end{tikzcd}\]
As we see in~(\ref{eq:final-seed}), the Pl\"ucker coordinate $p_{4589}$ belongs to
the final cluster associated with the maximal green sequence. Taking into account
the permutation $\sigma$, we see that its $F$-polynomial equals $DTF_{Q,7}$,
which is given by
\begin{eqnarray*}
 1  &+  y_7 + y_3y_7 +   y_7y_8 + y_7y_{10} +y_3y_7y_8+ y_3y_7y_{10}+y_7y_8y_{10} \\
     &+  y_3y_4y_7y_8 + y_3y_6y_7y_{10} + y_3y_7y_8y_{10} + y_7y_8y_{10}y_{11} \\
     &+ y_3y_4y_7y_8y_{10} + y_3y_6y_7y_8y_{10} +  y_3y_7y_8y_{10}y_{11} \\
     &+ y_3y_4y_6y_7y_8y_{10} + y_3y_6y_7y_8y_{10}y_{11} + y_3y_4y_7y_8y_{10}y_{11} \\
     &+ y_3y_4y_6y_7y_8 y_{10} y_{11} + y_3y_4y_6y_7^2y_8y_{10}y_{11}.
\end{eqnarray*}
By Theorem~\ref{thm:Nagao}, this is also the $F$-polynomial of the indecomposable injective
module $I_7$. A basis for $I_7$ formed by homogeneous vectors is given by equivalence classes 
of paths ending at the vertex $7$ (see \cite[Chapter 3, lemma 2.6]{assem06}), where the
degree of a path is the number of occurrences of diagonal arrows. As described in the proof of Lemma \ref{keylemmadt}, the homogenous submodules of $I_7$ correspond to the predecessor closed subsets of the poset ${\cal{P}}$ whose {\em Hasse diagram} looks as follows: 
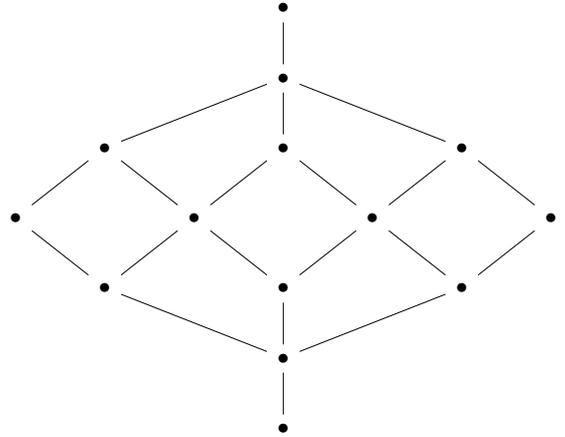
\[
\adjustbox{scale=0.8,center}
{\begin{tikzcd}
	&&& \bullet \\
	&&& \bullet \\
	& \bullet && \bullet && \bullet \\
	\bullet && \bullet && \bullet && \bullet \\
	& \bullet && \bullet && \bullet \\
	&&& \bullet \\
	&&& \bullet
	\arrow[no head, from=1-4, to=2-4]
	\arrow[no head, from=2-4, to=3-2]
	\arrow[no head, from=2-4, to=3-4]
	\arrow[no head, from=2-4, to=3-6]
	\arrow[no head, from=3-2, to=4-1]
	\arrow[no head, from=3-2, to=4-3]
	\arrow[no head, from=3-4, to=4-5]
	\arrow[no head, from=3-6, to=4-7]
	\arrow[no head, from=4-1, to=5-2]
	\arrow[no head, from=4-3, to=3-4]
	\arrow[no head, from=4-3, to=5-4]
	\arrow[no head, from=4-5, to=3-6]
	\arrow[no head, from=4-5, to=5-6]
	\arrow[no head, from=5-2, to=4-3]
	\arrow[no head, from=5-4, to=4-5]
	\arrow[no head, from=5-6, to=4-7]
	\arrow[no head, from=6-4, to=5-2]
	\arrow[no head, from=6-4, to=5-4]
	\arrow[no head, from=6-4, to=5-6]
	\arrow[no head, from=7-4, to=6-4]
\end{tikzcd}.
}
\]
The bullet in the $i$th row (read bottom to top) corresponds to homogeneous
submodules of degree $i-1$. For instance, the top row bullet corresponds to the 
whole injective module $I_7$,  the bottom one to the zero submodule and the bullet
on the second row corresponds to the simple socle $S_7$ of $I_7$.
\end{example}     

\bibliographystyle{amsplain}
\bibliography{references}

\providecommand{\bysame}{\leavevmode\hbox to3em{\hrulefill}\thinspace}
\providecommand{\MR}{\relax\ifhmode\unskip\space\fi MR }
% \MRhref is called by the amsart/book/proc definition of \MR.
\providecommand{\MRhref}[2]{%
  \href{http://www.ams.org/mathscinet-getitem?mr=#1}{#2}
}
\providecommand{\href}[2]{#2}
\begin{thebibliography}{10}

\bibitem{Amiot09}
Claire Amiot, \emph{Cluster categories for algebras of global dimension $2$ and
  quivers with potential}, Annales de l'institut {F}ourier \textbf{59} (2009),
  no.~6, 2525--2590.

\bibitem{assem06}
Ibrahim Assem, Andrzej Skowronski, and Daniel Simson, \emph{Elements of the
  representation theory of associative algebras: Techniques of representation
  theory}, London Mathematical Society Student Texts, Cambridge University
  Press, 2006.

\bibitem{BakshiKeller24v1}
Sarjick Bakshi and Bernhard Keller, \emph{{g}-vectors and {DT-F}-polynomials
  for {G}rassmannians}, arXiv:2410.01037 v1 [math.RT].

\bibitem{BaurBogdanic17}
Karin Baur and Dusko Bogdanic, \emph{Extensions between {C}ohen-{M}acaulay
  modules of {G}rassmannian cluster categories}, J. Algebraic Combin.
  \textbf{45} (2017), no.~4, 965--1000.

\bibitem{BialynickiBirula73}
A.~Bia{\l}ynicki-Birula, \emph{Some theorems on actions of algebraic groups},
  Ann. of Math. (2) \textbf{98} (1973), 480--497.

\bibitem{billey2000singular}
Sara Billey and V.~Lakshmibai, \emph{Singular loci of {S}chubert varieties},
  Progress in Mathematics, vol. 182, Birkh\"{a}user Boston, Inc., Boston, MA,
  2000.

\bibitem{BruestleDupontPerotin14}
Thomas Br\"{u}stle, Gr\'{e}goire Dupont, and Matthieu P\'{e}rotin, \emph{On
  maximal green sequences}, Int. Math. Res. Not. IMRN (2014), no.~16,
  4547--4586.

\bibitem{buan2009cluster}
A.~B. Buan, O.~Iyama, I.~Reiten, and J.~Scott, \emph{Cluster structures for
  2-{C}alabi-{Y}au categories and unipotent groups}, Compos. Math. \textbf{145}
  (2009), no.~4, 1035--1079.

\bibitem{CalderoChapoton06}
Philippe Caldero and Fr{\'e}d{\'e}ric Chapoton, \emph{Cluster algebras as
  {H}all algebras of quiver representations}, Comment. Math. Helv. \textbf{81}
  (2006), no.~3, 595--616.

\bibitem{casals2024}
Roger Casals, Eugene Gorsky, Mikhail Gorsky, Ian Le, Linhui Shen, and José
  Simental, \emph{Cluster structures on braid varieties}, arXiv:2207.11607
  (2022).

\bibitem{CanakciKingPressland24}
\.{I}lke \c{C}anak\c{c}\i, Alastair King, and Matthew Pressland, \emph{Perfect
  matching modules, dimer partition functions and cluster characters}, Adv.
  Math. \textbf{443} (2024), Paper No. 109570, 64.

\bibitem{CerulliKellerLabardiniPlamondon13}
Giovanni Cerulli~Irelli, Bernhard Keller, Daniel Labardini-Fragoso, and
  Pierre-Guy Plamondon, \emph{Linear independence of cluster monomials for
  skew-symmetric cluster algebras}, Compos. Math. \textbf{149} (2013), no.~10,
  1753--1764.

\bibitem{CrawleyBoeveyHolland98}
William Crawley-Boevey and Martin~P. Holland, \emph{Noncommutative deformations
  of {K}leinian singularities}, Duke Math. J. \textbf{92} (1998), no.~3,
  605--635.

\bibitem{dehy2008combinatorics}
Raika Dehy and Bernhard Keller, \emph{On the combinatorics of rigid objects in
  2-{C}alabi-{Y}au categories}, Int. Math. Res. Not. IMRN (2008), no.~11, Art.
  ID rnn029, 17.

\bibitem{DerksenWeymanZelevinsky10}
Harm Derksen, Jerzy Weyman, and Andrei Zelevinsky, \emph{Quivers with
  potentials and their representations {II}: {Applications to cluster
  algebras}}, J.~Amer.~Math.~Soc. \textbf{23} (2010), 749--790.

\bibitem{FockGoncharov09}
Vladimir~V. Fock and Alexander~B. Goncharov, \emph{Cluster ensembles,
  quantization and the dilogarithm}, Annales scientifiques de l'ENS \textbf{42}
  (2009), no.~6, 865--930.

\bibitem{fomin2002cluster}
Sergey Fomin and Andrei Zelevinsky, \emph{Cluster algebras. {I}.
  {F}oundations}, J. Amer. Math. Soc. \textbf{15} (2002), no.~2, 497--529.

\bibitem{fomin2003cluster}
\bysame, \emph{Cluster algebras. {II}. {F}inite type classification}, Invent.
  Math. \textbf{154} (2003), no.~1, 63--121.

\bibitem{fomin2007cluster}
\bysame, \emph{Cluster algebras. {IV}. {C}oefficients}, Compos. Math.
  \textbf{143} (2007), no.~1, 112--164.

\bibitem{Fraser16}
Chris Fraser, \emph{Quasi-homomorphisms of cluster algebras}, Adv. in Appl.
  Math. \textbf{81} (2016), 40--77.

\bibitem{FuKeller10}
Changjian Fu and Bernhard Keller, \emph{On cluster algebras with coefficients
  and 2-{C}alabi-{Y}au categories}, Trans. Amer. Math. Soc. \textbf{362}
  (2010), no.~2, 859--895.

\bibitem{GeissLeclercSchroeer08}
Christof Geiss, Bernard Leclerc, and Jan Schr\"{o}er, \emph{Partial flag
  varieties and preprojective algebras}, Ann. Inst. Fourier (Grenoble)
  \textbf{58} (2008), no.~3, 825--876.

\bibitem{GeissLeclercSchroer11}
Christof Gei\ss, Bernard Leclerc, and Jan Schr\"{o}er, \emph{Kac-{M}oody groups
  and cluster algebras}, Adv. Math. \textbf{228} (2011), no.~1, 329--433.

\bibitem{GekhtmanNakanishiRupel17}
Michael Gekhtman, Tomoki Nakanishi, and Dylan Rupel, \emph{Hamiltonian and
  {L}agrangian formalisms of mutations in cluster algebras and application to
  dilogarithm identities}, J. Integrable Syst. \textbf{2} (2017), no.~1,
  xyx005, 35 pp.

\bibitem{GoncharovShen2018}
Alexander Goncharov and Linhui Shen, \emph{Donaldson-{T}homas transformations
  of moduli spaces of {G}-local systems}, Adv. Math. \textbf{327} (2018),
  225--348.

\bibitem{gross2018canonical}
Mark Gross, Paul Hacking, Sean Keel, and Maxim Kontsevich, \emph{Canonical
  bases for cluster algebras}, J. Amer. Math. Soc. \textbf{31} (2018), no.~2,
  497--608.

\bibitem{happel1987derived}
Dieter Happel, \emph{On the derived category of a finite-dimensional algebra},
  Comment. Math. Helv. \textbf{62} (1987), no.~3, 339--389.

\bibitem{InoueIyamaKellerKunibaNakanishi13}
Rei Inoue, Osamu Iyama, Bernhard Keller, Atsuo Kuniba, and Tomoki Nakanishi,
  \emph{Periodicities of {T}-systems and {Y}-systems, dilogarithm identities,
  and cluster algebras {I}: type {$B_r$}}, Publ. Res. Inst. Math. Sci.
  \textbf{49} (2013), no.~1, 1--42.

\bibitem{InoueIyamaKellerKunibaNakanishi13a}
\bysame, \emph{Periodicities of {T}-systems and {Y}-systems, dilogarithm
  identities, and cluster algebras {II}: types {$C_r$}, {$F_4$}, and {$G_2$}},
  Publ. Res. Inst. Math. Sci. \textbf{49} (2013), no.~1, 43--85.

\bibitem{iyama2007higher}
Osamu Iyama, \emph{Higher-dimensional {A}uslander-{R}eiten theory on maximal
  orthogonal subcategories}, Adv. Math. \textbf{210} (2007), no.~1, 22--50.

\bibitem{JensenKingSu22}
Bernt~Tore Jensen, Alastair King, and Xiuping Su, \emph{Categorification and
  the quantum {G}rassmannian}, Adv. Math. \textbf{406} (2022), Paper No.
  108577, 29.

\bibitem{JensenKingSu16}
Bernt~Tore Jensen, Alastair~D. King, and Xiuping Su, \emph{A categorification
  of {G}rassmannian cluster algebras}, Proc. Lond. Math. Soc. (3) \textbf{113}
  (2016), no.~2, 185--212.

\bibitem{KashaevNakanishi11}
Rinat~M. Kashaev and Tomoki Nakanishi, \emph{Classical and quantum dilogarithm
  identities}, SIGMA Symmetry Integrability Geom. Methods Appl. \textbf{7}
  (2011), Paper 102, 29.

\bibitem{Keller17}
Bernhard Keller, \emph{Quiver mutation and combinatorial {DT}-invariants},
  corrected version of a submission to FPSAC 2013, arXiv:1709.03143 [math.CO].

\bibitem{QuiverMutation06}
\bysame, \emph{Quiver mutation in {J}ava and {J}ava{S}cript}, Interactive Java
  application available at the author's home page since 2006.

\bibitem{keller1996derived}
\bysame, \emph{Derived categories and their uses}, Handbook of algebra, {V}ol.
  1, Handb. Algebr., vol.~1, Elsevier/North-Holland, Amsterdam, 1996,
  pp.~671--701.

\bibitem{keller2011cluster}
\bysame, \emph{On cluster theory and quantum dilogarithm identities},
  Representations of algebras and related topics, EMS Ser. Congr. Rep., Eur.
  Math. Soc., Z\"{u}rich, 2011, pp.~85--116.

\bibitem{keller2012cluster}
\bysame, \emph{Cluster algebras and derived categories}, Derived categories in
  algebraic geometry, EMS Ser. Congr. Rep., Eur. Math. Soc., Z\"{u}rich, 2012,
  pp.~123--183.

\bibitem{KellerDemonet20}
Bernhard Keller and Laurent Demonet, \emph{A survey on maximal green
  sequences}, Representation theory and beyond, Contemp. Math., vol. 758, Amer.
  Math. Soc., Providence, RI, 2020, pp.~267--286.

\bibitem{keller2008acyclic}
Bernhard Keller and Idun Reiten, \emph{Acyclic {C}alabi-{Y}au categories},
  Compos. Math. \textbf{144} (2008), no.~5, 1332--1348, With an appendix by
  Michel Van den Bergh.

\bibitem{kontsevichsoibelman2008}
Maxim Kontsevich and Yan Soibelman, \emph{Stability structures, motivic
  {D}onaldson-{T}homas invariants and cluster transformations},
  arXiv:0811.2435.

\bibitem{krause2015krull}
Henning Krause, \emph{Krull-{S}chmidt categories and projective covers}, Expo.
  Math. \textbf{33} (2015), no.~4, 535--549.

\bibitem{lakshmibai1990multiplicities}
V.~Lakshmibai and J.~Weyman, \emph{Multiplicities of points on a {S}chubert
  variety in a minuscule {$G/P$}}, Adv. Math. \textbf{84} (1990), no.~2,
  179--208.

\bibitem{lakshmibai2007standard}
Venkatramani Lakshmibai and Komaranapuram~N. Raghavan, \emph{Standard monomial
  theory}, Encyclopaedia of Mathematical Sciences, vol. 137, Springer-Verlag,
  Berlin, 2008, Invariant theoretic approach, Invariant Theory and Algebraic
  Transformation Groups, 8.

\bibitem{MarshScott16}
R.~J. Marsh and J.~S. Scott, \emph{Twists of {P}l\"ucker coordinates as dimer
  partition functions}, Comm. Math. Phys. \textbf{341} (2016), no.~3, 821--884.
  \MR{3452273}

\bibitem{Muller16}
Greg Muller, \emph{The existence of a maximal green sequence is not invariant
  under quiver mutation}, Electron. J. Combin. \textbf{23} (2016), no.~2, Paper
  2.47, 23.

\bibitem{MullerSpeyer17}
Greg Muller and David~E. Speyer, \emph{The twist for positroid varieties},
  Proc. Lond. Math. Soc. (3) \textbf{115} (2017), no.~5, 1014--1071.
  \MR{3733558}

\bibitem{Nagao13}
Kentaro Nagao, \emph{Donaldson-{T}homas theory and cluster algebras}, Duke
  Math. J. \textbf{162} (2013), no.~7, 1313--1367.

\bibitem{Nakanishi11}
Tomoki Nakanishi, \emph{Dilogarithm identities for conformal field theories and
  cluster algebras: simply laced case}, Nagoya Math. J. \textbf{202} (2011),
  23--43.

\bibitem{Nakanishi2011}
\bysame, \emph{Periodicities in cluster algebras and dilogarithm identities},
  Representations of algebras and related topics, EMS Ser. Congr. Rep., Eur.
  Math. Soc., Z\"urich, 2011, pp.~407--443.

\bibitem{Nakanishi15}
\bysame, \emph{Quantum generalized cluster algebras and quantum dilogarithm
  functions of higher degrees}, Teoret. Mat. Fiz. \textbf{185} (2015), no.~3,
  460--470.

\bibitem{Nakanishi18}
\bysame, \emph{Rogers dilogarithms of higher degree and generalized cluster
  algebras}, J. Math. Soc. Japan \textbf{70} (2018), no.~4, 1269--1304.

\bibitem{Nakanishi21}
\bysame, \emph{Synchronicity phenomenon in cluster patterns}, J. Lond. Math.
  Soc. (2) \textbf{103} (2021), no.~3, 1120--1152.

\bibitem{Nakanishi24}
\bysame, \emph{Dilogarithm identities in cluster scattering diagrams}, Nagoya
  Math. J. \textbf{253} (2024), 1--22.

\bibitem{nakanishi2012tropical}
Tomoki Nakanishi and Andrei Zelevinsky, \emph{On tropical dualities in cluster
  algebras},  \textbf{565} (2012), 217--226.

\bibitem{OhPostnikovSpeyer15}
Suho Oh, Alexander Postnikov, and David~E. Speyer, \emph{Weak separation and
  plabic graphs}, Proc. Lond. Math. Soc. (3) \textbf{110} (2015), no.~3,
  721--754.

\bibitem{Palu09a}
Yann Palu, \emph{Grothendieck group and generalized mutation rule for
  2-{C}alabi-{Y}au triangulated categories}, J. Pure Appl. Algebra \textbf{213}
  (2009), no.~7, 1438--1449.

\bibitem{Postnikov06}
Alexander Postnikov, \emph{Total positivity, {G}rassmannians, and networks},
  arXiv preprint math/0609764 (2006).

\bibitem{Scott06}
Joshua~S. Scott, \emph{Grassmannians and cluster algebras}, Proc. London Math.
  Soc. (3) \textbf{92} (2006), no.~2, 345--380.

\bibitem{ShenWeng21}
Linhui Shen and Daping Weng, \emph{Cluster structures on double
  {B}ott-{S}amelson cells}, Forum Math. Sigma \textbf{9} (2021), Paper No. e66,
  89.

\bibitem{van2015calabi}
Michel Van~den Bergh, \emph{Calabi-{Y}au algebras and superpotentials}, Selecta
  Math. (N.S.) \textbf{21} (2015), no.~2, 555--603.

\bibitem{Weng23}
Daping Weng, \emph{F-polynomials of {D}onaldson-{T}homas transformations},
  arXiv:2303.03466.

\bibitem{Weng20}
\bysame, \emph{Donaldson-{T}homas transformation of double {B}ruhat cells in
  semisimple {L}ie groups}, Ann. Sci. \'Ec. Norm. Sup\'er. (4) \textbf{53}
  (2020), no.~2.

\bibitem{Weng21}
\bysame, \emph{Donaldson-{T}homas transformation of {G}rassmannian}, Adv. Math.
  \textbf{383} (2021), Paper No. 107721, 58.

\end{thebibliography}

\end{document}